\documentclass[11pt]{article}

\usepackage{amsmath}
\usepackage{amsthm}
\usepackage{amsfonts}

\usepackage{enumitem}

\usepackage{array}

\usepackage{setspace}

\usepackage{color}

\usepackage{hyperref}
\usepackage{graphicx}
\usepackage[numbers,sort&compress]{natbib}

\usepackage[margin=1.5 in]{geometry}

\usepackage{tikz}
\tikzstyle{vertex}=[circle, draw, inner sep=0pt, minimum size=3pt,fill=black]
\newcommand{\vertex}{\node[vertex]}

\newcommand{\T}[1]{\mathcal{T}_{#1}}
\newcommand{\C}[1]{\mathcal{C}_{#1}}

\newtheorem{theorem}{Theorem}[section]

\newtheorem{lemma}[theorem]{Lemma}
\newtheorem{corollary}[theorem]{Corollary}

\theoremstyle{definition}
\newtheorem{definition}[theorem]{Definition}

\newtheorem{result}[theorem]{Result}

\theoremstyle{remark}

\begin{document}

\title{Maximizing the mean subtree order}
\author{Lucas Mol %\thanks{Corresponding Author}\hspace{1mm}
and Ortrud R.\ Oellermann\thanks{Supported by an NSERC Grant CANADA, Grant number RGPIN-2016-05237}\\ University of Winnipeg\\
\small 515 Portage Avenue, Winnipeg, MB R3B 2E9\\
\small \href{mailto: l.mol@uwinnipeg.ca}{l.mol@uwinnipeg.ca}, \href{mailto: o.oellermann@uwinnipeg.ca}{o.oellermann@uwinnipeg.ca} }
\date{}

\maketitle

\begin{abstract}
This article focuses on properties and structures of trees with maximum mean subtree order in a given family; such trees are called \textit{optimal} in the family.  Our main goal is to describe the structure of optimal trees in $\T{n}$ and $\C{n}$, the families of all trees and caterpillars, respectively, of order $n$.    We begin by establishing a powerful tool called the \emph{Gluing Lemma}, which is used to prove several of our main results. In particular, we show that if $T$ is an optimal tree in $\T{n}$ or $\C{n}$ for $n\geq 4$, then every leaf of $T$ is adjacent to a vertex of degree at least $3$.  We also use the Gluing Lemma to answer an open question of Jamison, and to provide a conceptually simple proof of Jamison's result that the path $P_n$ has minimum mean subtree order among all trees of order $n$.  We prove that if $T$ is optimal in $\T{n}$, then the number of leaves in $T$ is $\mathrm{O}(\log_2 n)$, and that if $T$ is optimal in $\C{n}$, then the number of leaves in $T$ is $\mathrm{\Theta(\log_2 n)}$.  Along the way, we describe the asymptotic structure of optimal trees in several narrower families of trees.
\end{abstract}

%\begin{keyword}
%%% keywords here, in the form: keyword \sep keyword
%tree \sep subtree \sep mean order \sep caterpillar
%
%%% PACS codes here, in the form: \PACS code \sep code
%
%%% MSC codes here, in the form: \MSC code \sep code
%%% or \MSC[2008] code \sep code (2000 is the default)
%
%\MSC[2010] 05C05 \sep 05C30
%\end{keyword}

\section{Introduction}

The study of the mean order of the subtrees of a given tree was initiated by Jamison \cite{j1,j2}.  Given a tree $T$, the subtrees of $T$ form a meet-distributive lattice in which the subtrees of order $k$ are exactly the elements of rank $k$ (c.f.~\cite{j1}).  So the number of subtrees of order $k$ is the $k$th \emph{Whitney number} of the lattice of subtrees, and the mean subtree order gives a rough description of the shape of this lattice. We also point out that an interesting `inverse correlation' between the number of subtrees of a tree and the average distance of the tree (or its Wiener index) was established in \cite{Wagner2007}.

Jamison~\cite{j1} proved that among all trees of a fixed order $n$, the path $P_n$ has minimum mean subtree order $\tfrac{n+2}{3}$. However, the problem of describing the tree(s) of a given order having maximum mean subtree order remains largely open, despite the fact that other open problems from \cite{j1} have been solved~\cite{vw,ww,ww1}.  We will call a tree of maximum mean subtree order in a given family of trees an {\em optimal tree} for this family.  In this paper, we focus on determining properties and structures of optimal trees for several families of trees.  Families of particular interest include the family $\T{n}$ of all trees of a given order $n$, and the family $\C{n}$ of all caterpillars of order $n$.
Jamison conjectured that for all $n$, every optimal tree in $\T{n}$ is a caterpillar.  We refer to this as the {\em Caterpillar Conjecture}, and though we do not resolve this conjecture, we make significant progress on determining the structure of those trees that are optimal in $\T{n}$ or $\C{n}$.  
%We prove that if $T$ is optimal in $\T{n}$ (or $\C{n}$) for $n\geq 4$, then every leaf of $T$ is adjacent to a vertex of degree at least $3$.  Jamison~\cite{j1} gave trees of order $n$ whose mean subtree order is asymptotically $n-\mathrm{o}(n)$, and we strengthen this result by finding trees (in fact, caterpillars) of order $n$ whose mean subtree order is asymptotically $n-\mathrm{O}(\log_2 n)$.  This leads to a proof that if $T$ is optimal in $\T{n}$ (or $\C{n}$), then $T$ has at most $\mathrm{O}(\log_2 n)$ leaves.  Finally, we show that if $T$ is optimal in $\C{n}$, then the number of leaves of $T$ is $\Theta(\log_2 n)$.

We now give a description of the layout of the article.  In Section~\ref{BackgroundSection}, we provide the necessary background for our work.  In Section~\ref{LemmaSection}, we establish a key result which we refer to as the \emph{Gluing Lemma}. Let $G$ and $H$ be disjoint graphs and let $v$ be a fixed vertex of $G$.  The graph obtained by identifying $v$ with some vertex of $H$ is called a graph obtained from $G$ and $H$ by \emph{gluing} $v$ to a vertex of $H$.  The Gluing Lemma states that for a given tree $Q$ with fixed vertex $v$, among all trees obtained by gluing $v$ to a vertex of a path $P$ (disjoint from $Q$), the maximum mean subtree order is obtained by gluing $v$ to a central vertex of $P$.   We use the Gluing Lemma to describe the structure of the optimal tree in several families, including the family of all trees of order $n$ with the same \emph{core} (see Definition~\ref{core}).  Next, we show that for $n\geq 4$, if $T$ is optimal in $\T{n}$, then every leaf of $T$ is adjacent to a vertex of degree at least $3$.  Finally, we show that every tree, not isomorphic to a path, has a \emph{standard $1$-associate} (defined in the paragraph following Theorem \ref{ShortLimbs}) with smaller mean subtree order. This answers a question of Jamison~\cite{j1} in the affirmative.

In Section~\ref{DoubleStarSection} and Section~\ref{BridgeSection}, we study the mean subtree order of two families of trees introduced by Jamison~\cite{j1}, which he called batons and bridges (see Section~\ref{BackgroundSection} for definitions). Jamison demonstrated sequences of batons and bridges of order $n$ whose mean subtree order grows as $n - \mathrm{o}(n)$.  In Section \ref{DoubleStarSection}, we show that among all subdivided double stars of a fixed order and with a fixed even (and sufficiently large) number of leaves, the batons are optimal.  Then we describe the asymptotic structure of the optimal batons among all batons of order $n,$ and show that the maximum mean subtree order among all batons of order $n$ is asymptotically $n-\mathrm{O}(\log_2 n)$.  This leads to an $\mathrm{O}(\log_2 n)$ upper bound on the number of leaves in any optimal tree in $\T{n}$ or $\C{n}$.  In Section \ref{BridgeSection}, we consider the problem of finding the optimal tree among all bridges of a fixed order.  The asymptotic structure of optimal bridges is described, and it is shown that the mean subtree order of an optimal bridge of a given order is significantly lower than the mean subtree order of an optimal baton of the same order.

Finally, in Section~\ref{CaterpillarLeaves}, we prove an upper bound on the mean subtree order of a caterpillar in terms of its number of leaves.  This leads to a proof that if $T$ is optimal in $\C{n}$, then the number of leaves of $T$ is $\Theta(\log_2 n)$.  Our work in this section also extends to several other families of trees.

\section{Background and Preliminaries}\label{BackgroundSection}

We use standard graph theoretic terminology throughout.  For a graph $G$, we use $V(G)$ to denote the vertex set of $G$. A vertex $v\in V(G)$ is called a \textit{leaf} of $G$ if it has degree at most $1$, and is called an \textit{internal vertex} of $G$ otherwise.  The \textit{eccentricity} of $v$ in $G$ is defined to be the greatest distance between $v$ and any other vertex of $G.$  The \textit{centre} of $G$ is the set of vertices of minimum \textit{eccentricity}.  A vertex belonging to the centre of a graph $G$ will be called a \textit{central vertex} of $G.$  In a tree $T$, the centre contains at most two vertices, and can be found by recursively deleting all leaves from $T$ until either one or two vertices remain; these remaining vertices form the centre of $T$.

For a given tree $T$ of order $n$, let $a_k(T)$ denote the number of subtrees of order $k$ in $T$. Then the generating polynomial for the number of subtrees of $T$, which we call the {\em subtree polynomial of $T$}, is given by
\[ \Phi_T(x) = \sum_{k=1}^n a_k(T)x^k\]
and the {\em (global) mean subtree order} of $T$ (sometimes called the {\em global mean} for short) is given by the logarithmic derivative
\[M_T=\frac{\Phi_T'(1)}{\Phi_T(1)}.\]

If $v$ is a vertex of $T$, let $a_k(T;v)$ denote the number of subtrees of $T$ of order $k$ containing $v$. Then the generating polynomial for the number of subtrees of $T$ containing $v$, which we refer to as the {\em local subtree polynomial of $T$ at $v$}, is given by
\[
\Phi_T(v;x)=\sum_{k=1}^{n}a_k(T;v)x^k
\]
and the mean order of the subtrees of $T$ containing $v$, called the {\em local mean of $T$ at $v$}, is given by the logarithmic derivative $M_{T, v}=\Phi_T'(v;1)/\Phi_T(v;1)$.

As a straightforward first example, the reader may wish to verify the following result from \cite{j1}, which can be proven using basic counting arguments.  We will use these formulae in Section \ref{LemmaSection}.

\begin{result}\label{paths}
Let $P_n:u_1\dots u_n$ be a path of order $n$.  For any $s\in\{1,\dots,n\}$:
\begin{align*}
\Phi_{P_n}(1) &= \tbinom{n+1}{2} &\Phi_{P_n}(u_s;1)&=s(n-s+1)\\
\Phi_{P_n}'(1) &= \tbinom{n+2}{3} &\Phi'_{P_n}(u_s;1)&=s(n-s+1)\tfrac{n+1}{2}.
 \end{align*}
\end{result}

Jamison \cite{j1} established the following relationship between the global mean and local mean subtree orders.

\begin{result} \label{global_local}
For any tree $T$ and vertex $v$ of $T$, $M_T \le M_{T,v}$, with equality if and only if $T$ is the trivial tree $K_1$.
\end{result}
\noindent
It was subsequently shown in \cite{ww} that $M_{T,v}\leq 2M_T$, and that the maximum local mean of a tree occurs at either a leaf or a vertex of degree $2$ (and that both cases are possible), thereby answering two open questions from \cite{j1}.

For a subtree $H$ of a tree $T$, the mean order of the subtrees containing $H$ is denoted by $M_{T,H}$ and $T/H$ denotes the tree obtained from $T$ by contracting $H$ to a single vertex. The following results that appeared in \cite{j1} will be useful in subsequent discussions.

\begin{result} \label{nested_trees}
Let $R$ and $S$ be subtrees of a tree $T$ of orders $r$ and $s$, respectively, such that $R$ is a subtree of $S$.  Then
\begin{enumerate}
\item $M_{T,S} = M_{T/R, S/R} + r-1$, and
\item $M_{T,R}<M_{T,S}\le M_{T,R} + \frac{s-r}{2}$ whenever $R\neq S.$
\end{enumerate}
\end{result}

The {\em density} of a tree $T$ is defined by $\mathrm{den}(T)=M_T/|V(T)|$, and equals the probability that a randomly chosen vertex belongs to a randomly chosen subtree of $T$. Since $M_{P_n}=\frac{n+2}{3}$, and this mean is smallest among all trees of order $n,$ the density of every tree exceeds $1/3$. It is natural to ask whether there is a constant $c <1$ which serves as upper bound on the density of all trees.  Jamison~\cite{j1} gave two families of trees whose densities are asymptotically $1$, thereby answering the aforementioned question in the negative.  We describe these constructions below as they motivate some of our work.

For integers $s\geq 1$ and $t\geq 0$, an $(s,t)$-{\em baton} is the tree of order $2s+t+2$ obtained by joining the central vertices of two stars $K_{1,s}$ by a path of order $t$ (if $t=0$, the central vertices are simply joined by an edge).  The density of the $(k,k^2)$-baton approaches $1$ as $k\rightarrow \infty$, so that there are batons of density arbitrarily close to $1.$    The batons form a subclass of the {\em subdivided double stars}. For positive integers $n,$ $r,$ and $s$, the {\em subdivided double star} $D_n(r,s)$ is the tree obtained by joining a vertex of degree $r$ in the star $K_{1,r}$ and a vertex of degree $s$ in the star $K_{1,s}$ by a path of order $n-r-s-2$.  We call this path of order $n-r-s-2$ the {\em interior path} of the subdivided double star.  Note that $D_n(s,s)$ is a baton, also called a \textit{balanced} subdivided double star.

The second class of high density trees, described by Jamison, have exactly two vertices of degree $3$ and all others of degree $1$ or $2$.  For $s\geq 1$ and $t\geq 0,$ an $(s,t)$-{\em bridge}, denoted by $B(s,t)$, is obtained by joining the central vertices of two paths of order $2s+1$ by a path of order $t$ (if $t=0$, then the vertices are simply joined by an edge). Thus $B(s,t)$ has order $4s+t+2$.  As for batons, there are bridges of density arbitrarily close to $1$; in particular, the density of $B(k^2,k^3)$ approaches $1$ as $k\rightarrow\infty$.  Bridges also belong to a larger family with the same basic structure.  A {\em stickman} is a tree $T$ obtained from two paths $P$ and $Q$ of order at least $3$ by joining some internal vertex of $P$ and some internal vertex of $Q$ by a path $H$ (if $H$ is empty, the vertices are simply joined by an edge).  We call $H$ the {\em interior path} of the stickman; this is the path that remains after we delete the vertices of $P$ and $Q$.

The above examples of bridges and batons demonstrate that trees with high density need not have many vertices of large degree. Indeed it was shown in \cite{j1} that if $T_k$ is a sequence of trees such that $\mathrm{den}(T_k) \rightarrow 1$, then the ratio of the number of vertices of degree 2 of $T_k$ to $|V(T_k)|$ approaches $1$.  This result is a fairly direct corollary of the following result from \cite{j1}, which we apply later on.

\begin{result}\label{LeavesLemma}
If $T$ is a tree of order $n\geq 3$ with $\ell$ leaves, then
\[
M_T<n-\tfrac{\ell}{2}.
\]
\end{result}

The work of Haslegrave \cite{has} provides a similar upper bound on the mean subtree order of $T$ in terms of the number of leaves of $T$ and the number of \textit{twigs} of $T.$  A vertex $v$ of a tree $T$ is called a \textit{twig} if $\deg(v)\geq 2$ and at least $\deg(v)-1$ of its neighbours are leaves.  Note that $T$ is a caterpillar if and only if it has at most two twigs.  The following result was proven implicitly in \cite{has}; specifically it follows from Lemma 3 and the proof of Lemma 4 there.

\begin{result}\label{HasTwigs}
Let $T$ be a tree of order at least $4$ in which every twig has degree at least $3.$  If $T$ has $t$ twigs, $\ell_1$ leaves adjacent to twigs, and $\ell_2$ leaves not adjacent to twigs, then
\[
M_T<n-\tfrac{29}{45}t-\tfrac{17}{45}\ell_1-\tfrac{1}{2}\ell_2\leq n-\tfrac{7}{5}t.
\]
\end{result}

Note that while Lemma 4 of \cite{has} is stated only for \textit{series-reduced} trees (those in which every internal vertex has degree at least $3$), all that is actually required to reach the first inequality of  Result \ref{HasTwigs} is that every \textit{twig} has degree at least $3$.  The second inequality of Result \ref{HasTwigs} follows from the basic facts that $\ell_1\geq 2t$ (since every twig has degree at least $3$) and $\ell_2\geq 0.$

An {\em aster} is a tree with at most one vertex of degree exceeding $2$.  We say that an aster $T$ is {\em astral over $v\in V(T)$} if $T$ is a path or if $v$ has degree greater than $2$ in $T$.  An aster is {\em balanced} if it is a path or if any two of the paths emanating from the vertex of maximum degree differ in order by at most 1.  It was shown in \cite{j1} that if $T$ has order $n$ and is astral over $v$, then the local mean subtree order of $T$ at $v$ is $\tfrac{n+1}{2},$ and that among all asters of a given order, the stars have maximum global mean subtree order.

We give a brief summary of other work done on the mean subtree order problem.
Vince and Wang \cite{vw} showed that if $T$ is a series-reduced tree, then $\frac{1}{2} \le \mathrm{den}(T) < \frac{3}{4}$, and that both bounds are asymptotically sharp.  Moreover, Haslegrave \cite{has} demonstrated necessary and sufficient conditions for a sequence of distinct series-reduced trees to have density approaching either bound.  It was shown in \cite{ww1} that for almost every tree $T$, there is a tree $T'$ of the same order such that $T\not\cong T'$ and $M_T=M_{T'}.$

Before we proceed with our main results, we discuss computational results supporting the Caterpillar Conjecture, which is the subject of the last remaining open problem from~\cite{j1}.  For all $n\leq 24$, we used McKay and Piperno's program \emph{nauty}~\cite{mp} to enumerate all nonisomorphic trees of order $n$, and then implemented the linear time-algorithm of Yan and Yeh~\cite{yy} to compute the subtree polynomial, and in turn the mean subtree order, of all of these trees.  The optimal tree in $\T{n}$ for $n\leq 15$ is indeed the one demonstrated (for $n\leq 10$) or conjectured (for $11\leq n\leq 15$) to be optimal in~\cite{j1} (and illustrated there), and the optimal tree in $\T{n}$ for $16\leq n\leq 24$ is illustrated in Figure~\ref{OptimalTrees}.  Note that the optimal trees are all caterpillars, and that they appear to be rather baton-like.  This motivates our in-depth study of batons (and subdivided double stars in general).

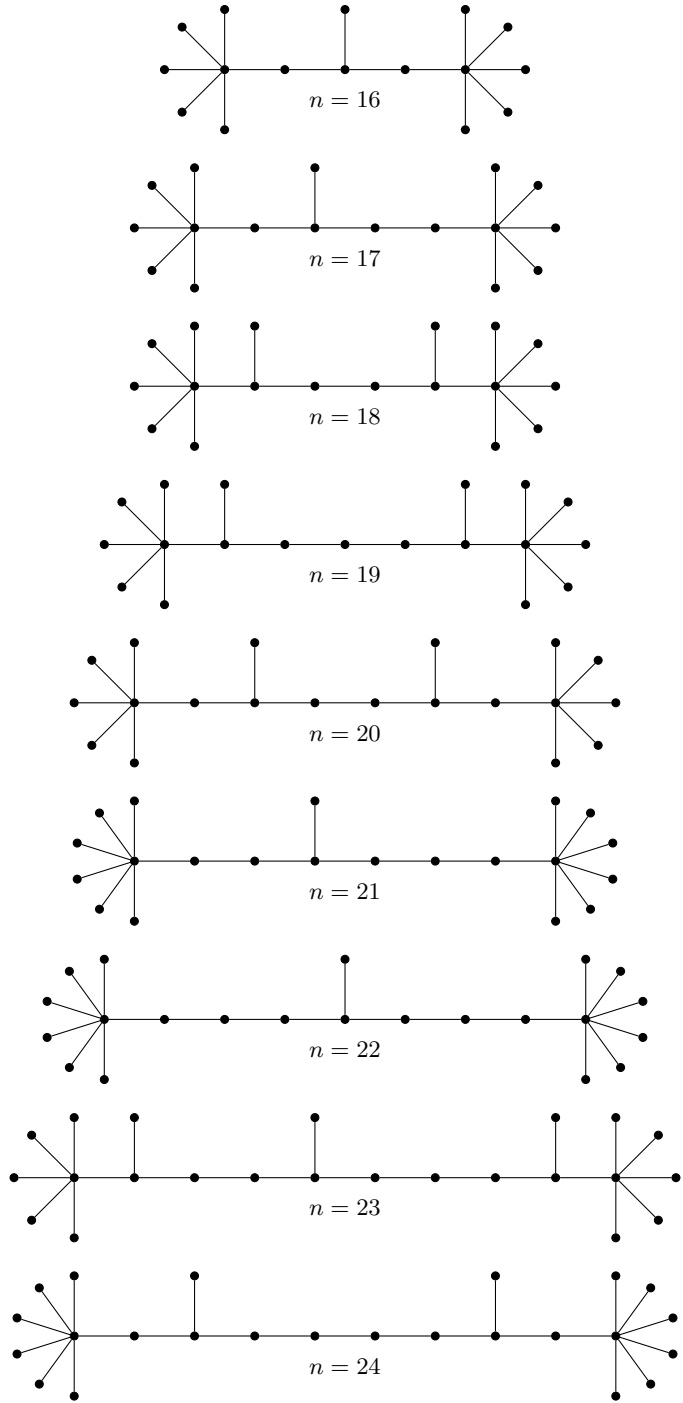
\begin{figure}
\centering{
\begin{tikzpicture}[scale=0.8]
\vertex (1) at (0,0) {};
\vertex (2) at (1,0) {};
\vertex (3) at (2,0) {};
\vertex (4) at (3,0) {};
\vertex (5) at (4,0) {};
\vertex (6) at (0,1) {};
\vertex (7) at ({-sqrt(2)/2},{sqrt(2)/2}) {};
\vertex (8) at (-1,0) {};
\vertex (9) at ({-sqrt(2)/2},{-sqrt(2)/2}) {};
\vertex (10) at (0,-1) {};
\vertex (11) at (2,1) {};
\vertex (12) at (4,1) {};
\vertex (13) at ({4+sqrt(2)/2},{sqrt(2)/2}) {};
\vertex (14) at (5,0) {};
\vertex (15) at ({4+sqrt(2)/2},-{sqrt(2)/2}) {};
\vertex (16) at (4,-1) {};
\path
(1) edge (2)
(2) edge (3)
(3) edge (4)
(4) edge (5)
(1) edge (6)
(1) edge (7)
(1) edge (8)
(1) edge (9)
(1) edge (10)
(3) edge (11)
(5) edge (12)
(5) edge (13)
(5) edge (14)
(5) edge (15)
(5) edge (16);
\draw (2,-0.5) node {\footnotesize $n=16$};
\end{tikzpicture}\\\vspace{10pt}

\begin{tikzpicture}[scale=0.8]
\vertex (1) at (0,0) {};
\vertex (2) at (1,0) {};
\vertex (3) at (2,0) {};
\vertex (4) at (3,0) {};
\vertex (5) at (4,0) {};
\vertex (6) at (0,1) {};
\vertex (7) at ({-sqrt(2)/2},{sqrt(2)/2}) {};
\vertex (8) at (-1,0) {};
\vertex (9) at ({-sqrt(2)/2},{-sqrt(2)/2}) {};
\vertex (10) at (0,-1) {};
\vertex (11) at (2,1) {};
\vertex (12) at (5,1) {};
\vertex (13) at ({5+sqrt(2)/2},{sqrt(2)/2}) {};
\vertex (14) at (6,0) {};
\vertex (15) at ({5+sqrt(2)/2},-{sqrt(2)/2}) {};
\vertex (16) at (5,-1) {};
\vertex (17) at (5,0) {};
\path
(1) edge (2)
(2) edge (3)
(3) edge (4)
(4) edge (5)
(5) edge (17)
(1) edge (6)
(1) edge (7)
(1) edge (8)
(1) edge (9)
(1) edge (10)
(3) edge (11)
(17) edge (12)
(17) edge (13)
(17) edge (14)
(17) edge (15)
(17) edge (16);
\draw (2.5,-0.5) node {\footnotesize $n=17$};
\end{tikzpicture}\\\vspace{10pt}

\begin{tikzpicture}[scale=0.8]
\vertex (1) at (0,0) {};
\vertex (2) at (1,0) {};
\vertex (3) at (2,0) {};
\vertex (4) at (3,0) {};
\vertex (5) at (4,0) {};
\vertex (6) at (0,1) {};
\vertex (7) at ({-sqrt(2)/2},{sqrt(2)/2}) {};
\vertex (8) at (-1,0) {};
\vertex (9) at ({-sqrt(2)/2},{-sqrt(2)/2}) {};
\vertex (10) at (0,-1) {};
\vertex (11) at (1,1) {};
\vertex (12) at (5,1) {};
\vertex (13) at ({5+sqrt(2)/2},{sqrt(2)/2}) {};
\vertex (14) at (6,0) {};
\vertex (15) at ({5+sqrt(2)/2},-{sqrt(2)/2}) {};
\vertex (16) at (5,-1) {};
\vertex (17) at (5,0) {};
\vertex (18) at (4,1) {};
\path
(1) edge (2)
(2) edge (3)
(3) edge (4)
(4) edge (5)
(5) edge (17)
(1) edge (6)
(1) edge (7)
(1) edge (8)
(1) edge (9)
(1) edge (10)
(2) edge (11)
(5) edge (18)
(17) edge (12)
(17) edge (13)
(17) edge (14)
(17) edge (15)
(17) edge (16);
\draw (2.5,-0.5) node {\footnotesize $n=18$};
\end{tikzpicture}\\\vspace{10pt}

\begin{tikzpicture}[scale=0.8]
\vertex (1) at (0,0) {};
\vertex (2) at (1,0) {};
\vertex (3) at (2,0) {};
\vertex (4) at (3,0) {};
\vertex (5) at (4,0) {};
\vertex (6) at (0,1) {};
\vertex (7) at ({-sqrt(2)/2},{sqrt(2)/2}) {};
\vertex (8) at (-1,0) {};
\vertex (9) at ({-sqrt(2)/2},{-sqrt(2)/2}) {};
\vertex (10) at (0,-1) {};
\vertex (11) at (1,1) {};
\vertex (12) at (6,1) {};
\vertex (13) at ({6+sqrt(2)/2},{sqrt(2)/2}) {};
\vertex (14) at (6,0) {};
\vertex (15) at ({6+sqrt(2)/2},-{sqrt(2)/2}) {};
\vertex (16) at (6,-1) {};
\vertex (17) at (5,0) {};
\vertex (18) at (5,1) {};
\vertex (19) at (7,0) {};
\path
(1) edge (2)
(2) edge (3)
(3) edge (4)
(4) edge (5)
(5) edge (17)
(1) edge (6)
(1) edge (7)
(1) edge (8)
(1) edge (9)
(1) edge (10)
(2) edge (11)
(17) edge (18)
(14) edge (12)
(14) edge (13)
(14) edge (19)
(14) edge (15)
(14) edge (16)
(17) edge (14);
\draw (3,-0.5) node {\footnotesize $n=19$};
\end{tikzpicture}\\\vspace{10pt}

\begin{tikzpicture}[scale=0.8]
\vertex (1) at (0,0) {};
\vertex (2) at (1,0) {};
\vertex (3) at (2,0) {};
\vertex (4) at (3,0) {};
\vertex (5) at (4,0) {};
\vertex (6) at (5,0) {};
\vertex (7) at (6,0) {};
\vertex (8) at (7,0) {};
\vertex (9) at (0,1) {};
\vertex (10) at ({-sqrt(2)/2},{sqrt(2)/2}) {};
\vertex (11) at (-1,0) {};
\vertex (12) at ({-sqrt(2)/2},{-sqrt(2)/2}) {};
\vertex (13) at (0,-1) {};
\vertex (14) at (2,1) {};
\vertex (15) at (5,1) {};
\vertex (16) at (7,1) {};
\vertex (17) at ({7+sqrt(2)/2},{sqrt(2)/2}) {};
\vertex (18) at (8,0) {};
\vertex (19) at ({7+sqrt(2)/2},-{sqrt(2)/2}) {};
\vertex (20) at (7,-1) {};
\path
(1) edge (2)
(2) edge (3)
(3) edge (4)
(4) edge (5)
(5) edge (6)
(6) edge (7)
(7) edge (8)
(1) edge (9)
(1) edge (10)
(1) edge (11)
(1) edge (12)
(1) edge (13)
(3) edge (14)
(6) edge (15)
(8) edge (16)
(8) edge (17)
(8) edge (18)
(8) edge (19)
(8) edge (20);
\draw (3.5,-0.5) node {\footnotesize $n=20$};
\end{tikzpicture}\\\vspace{10pt}

\begin{tikzpicture}[scale=0.8]
\vertex (1) at (0,0) {};
\vertex (2) at (1,0) {};
\vertex (3) at (2,0) {};
\vertex (4) at (3,0) {};
\vertex (5) at (4,0) {};
\vertex (6) at (5,0) {};
\vertex (7) at (6,0) {};
\vertex (8) at (7,0) {};
\vertex (9) at (0,1) {};
\vertex (10) at (-0.58,0.8) {};
\vertex (11) at (-0.95,0.3) {};
\vertex (12) at (-0.95,-0.3) {};
\vertex (13) at (-0.58,-0.8) {};
\vertex (14) at (0,-1) {};
\vertex (15) at (3,1) {};
\vertex (16) at (7,1) {};
\vertex (17) at (7.58,0.8) {};
\vertex (18) at (7.95,0.3) {};
\vertex (19) at (7.95,-0.3) {};
\vertex (20) at (7.58,-0.8) {};
\vertex (21) at (7,-1) {};
\path
(1) edge (2)
(2) edge (3)
(3) edge (4)
(4) edge (5)
(5) edge (6)
(6) edge (7)
(7) edge (8)
(1) edge (9)
(1) edge (10)
(1) edge (11)
(1) edge (12)
(1) edge (13)
(1) edge (14)
(4) edge (15)
(8) edge (16)
(8) edge (17)
(8) edge (18)
(8) edge (19)
(8) edge (20)
(8) edge (21);
\draw (3.5,-0.5) node {\footnotesize $n=21$};
\end{tikzpicture}\\\vspace{10pt}

\begin{tikzpicture}[scale=0.8]
\vertex (1) at (0,0) {};
\vertex (2) at (1,0) {};
\vertex (3) at (2,0) {};
\vertex (4) at (3,0) {};
\vertex (5) at (4,0) {};
\vertex (6) at (5,0) {};
\vertex (7) at (6,0) {};
\vertex (8) at (7,0) {};
\vertex (9) at (8,0) {};
\vertex (10) at (0,1) {};
\vertex (11) at (-0.58,0.8) {};
\vertex (12) at (-0.95,0.3) {};
\vertex (13) at (-0.95,-0.3) {};
\vertex (14) at (-0.58,-0.8) {};
\vertex (15) at (0,-1) {};
\vertex (16) at (4,1) {};
\vertex (17) at (8,1) {};
\vertex (18) at (8.58,0.8) {};
\vertex (19) at (8.95,0.3) {};
\vertex (20) at (8.95,-0.3) {};
\vertex (21) at (8.58,-0.8) {};
\vertex (22) at (8,-1) {};
\path
(1) edge (2)
(2) edge (3)
(3) edge (4)
(4) edge (5)
(5) edge (6)
(6) edge (7)
(7) edge (8)
(8) edge (9)
(1) edge (10)
(1) edge (11)
(1) edge (12)
(1) edge (13)
(1) edge (14)
(1) edge (15)
(5) edge (16)
(9) edge (17)
(9) edge (18)
(9) edge (19)
(9) edge (20)
(9) edge (21)
(9) edge (22);
\draw (4,-0.5) node {\footnotesize $n=22$};
\end{tikzpicture}\\\vspace{10pt}

\begin{tikzpicture}[scale=0.8]
\vertex (1) at (0,0) {};
\vertex (2) at (1,0) {};
\vertex (3) at (2,0) {};
\vertex (4) at (3,0) {};
\vertex (5) at (4,0) {};
\vertex (6) at (5,0) {};
\vertex (7) at (6,0) {};
\vertex (8) at (7,0) {};
\vertex (9) at (8,0) {};
\vertex (10) at (9,0) {};
\vertex (11) at (0,1) {};
\vertex (12) at ({-sqrt(2)/2},{sqrt(2)/2}) {};
\vertex (13) at (-1,0) {};
\vertex (14) at ({-sqrt(2)/2},{-sqrt(2)/2}) {};
\vertex (15) at (0,-1) {};
\vertex (16) at (1,1) {};
\vertex (17) at (4,1) {};
\vertex (18) at (8,1) {};
\vertex (19) at (9,1) {};
\vertex (20) at ({9+sqrt(2)/2},{sqrt(2)/2}) {};
\vertex (21) at (10,0) {};
\vertex (22) at ({9+sqrt(2)/2},-{sqrt(2)/2}) {};
\vertex (23) at (9,-1) {};
\path
(1) edge (2)
(2) edge (3)
(3) edge (4)
(4) edge (5)
(5) edge (6)
(6) edge (7)
(7) edge (8)
(8) edge (9)
(9) edge (10)
(1) edge (11)
(1) edge (12)
(1) edge (13)
(1) edge (14)
(1) edge (15)
(2) edge (16)
(5) edge (17)
(9) edge (18)
(10) edge (19)
(10) edge (20)
(10) edge (21)
(10) edge (22)
(10) edge (23);
\draw (4.5,-0.5) node {\footnotesize $n=23$};
\end{tikzpicture}\\\vspace{10pt}

\begin{tikzpicture}[scale=0.8]
\vertex (1) at (0,0) {};
\vertex (2) at (1,0) {};
\vertex (3) at (2,0) {};
\vertex (4) at (3,0) {};
\vertex (5) at (4,0) {};
\vertex (6) at (5,0) {};
\vertex (7) at (6,0) {};
\vertex (8) at (7,0) {};
\vertex (9) at (8,0) {};
\vertex (10) at (9,0) {};
\vertex (11) at (0,1) {};
\vertex (12) at (-0.58,0.8) {};
\vertex (13) at (-0.95,0.3) {};
\vertex (14) at (-0.95,-0.3) {};
\vertex (15) at (-0.58,-0.8) {};
\vertex (16) at (0,-1) {};
\vertex (17) at (2,1) {};
\vertex (18) at (7,1) {};
\vertex (19) at (9,1) {};
\vertex (20) at (9.58,0.8) {};
\vertex (21) at (9.95,0.3) {};
\vertex (22) at (9.95,-0.3) {};
\vertex (23) at (9.58,-0.8) {};
\vertex (24) at (9,-1) {};
\path
(1) edge (2)
(2) edge (3)
(3) edge (4)
(4) edge (5)
(5) edge (6)
(6) edge (7)
(7) edge (8)
(8) edge (9)
(9) edge (10)
(1) edge (11)
(1) edge (12)
(1) edge (13)
(1) edge (14)
(1) edge (15)
(1) edge (16)
(3) edge (17)
(8) edge (18)
(10) edge (19)
(10) edge (20)
(10) edge (21)
(10) edge (22)
(10) edge (23)
(10) edge (24);
\draw (4.5,-0.5) node {\footnotesize $n=24$};
\end{tikzpicture}\\\vspace{10pt}
}
\caption{The optimal tree in $\T{n}$ for $16\leq n\leq 24.$}
\label{OptimalTrees}
\end{figure}

\section{The Gluing Lemma}\label{LemmaSection}

Let $Q$ be any tree with root $v$ and let $u$ be a vertex of a path $P$ of order at least $3$. In this section we show that among all trees obtained from the disjoint union of $P$ and $Q$ by  gluing $v$ to some vertex $u$ of $P$ (i.e.\ by identifying $v$ and $u$), the maximum mean subtree order is obtained if $u$ is a central vertex of $P$. This result is used to describe optimal trees for several families, to prove a necessary condition on any optimal tree among all trees of a fixed order, and to answer an open problem from \cite{j1}.

\begin{lemma}[Gluing Lemma]
Fix a natural number $n \ge 3$ and a tree $Q$ of order at least $2$ having root vertex $v$. Let $P:u_1\dots u_n$ be a path of order $n$. For $s \in \{1, \dots, n\}$, let $T_s$ be the tree obtained from the disjoint union of $P$ and $Q$ by gluing $v$ to $u_s$.  Among all such trees $T_s$, the tree $T_{\left\lfloor \frac{n+1}{2}\right\rfloor}$ is the optimal tree.  In other words, the maximum mean subtree order of trees constructed this way occurs when $v$ is glued to a central vertex of $P$.
\end{lemma}

\begin{proof} We may assume $s \le \frac{n+1}{2}$.  The subtrees of $T_s$ can be partitioned into three types:
\begin{itemize}
\item Those that lie in $P$ but do not contain $u_s$.  These are counted by the polynomial $\Phi_P(x)-\Phi_P(u_s;x)$.
\item Those that lie in $Q$ but do not contain $v$.  These are counted by the polynomial $\Phi_Q(x)-\Phi_Q(v;x)$.
\item Those that contain the glued vertex.  These are counted by the polynomial $\displaystyle\frac{\Phi_P(u_s;x)\Phi_Q(v;x)}{x}$.
\end{itemize}
Thus,
\begin{align}\label{Phi}
\Phi_{T_s}(x)= \Phi_P(x)-\Phi_P(u_s;x)+\Phi_Q(x)-\Phi_Q(v;x) + \frac{\Phi_P(u_s;x)\Phi_Q(v;x)}{x}.
\end{align}
Evaluating the derivative gives
\begin{align}\label{PhiPrime}
\begin{split}
\Phi'_{T_s}(x)&=  \Phi'_P(x)-\Phi'_P(u_s;x)+\Phi'_Q(x) -\Phi'_Q(v;x)\\
& \ \ \ \ + \frac{\Phi'_P(u_s;x)\Phi_Q(v;x)}{x}+\Phi_P(u_s;x)\left[\frac{x\Phi'_Q(v;x)-\Phi_Q(v;x)}{x^2}\right].
\end{split}
\end{align}
Evaluating (\ref{Phi}) and (\ref{PhiPrime}) at $1$ yields
\begin{align}\label{TsAt1}
\Phi_{T_s}(1)&= \Phi_P(1)+\Phi_Q(1)-\Phi_Q(v;1) + \Phi_P(u_s;1)\left[\Phi_Q(v;1)-1\right],
\end{align}
and
\begin{align}\label{TsprimeAt1}
\begin{split}
\Phi'_{T_s}(1)&=  \Phi'_P(1)+\Phi'_Q(1) -\Phi'_Q(v;1)\\
& \ \ \ \ +\Phi'_P(u_s;1)\left[\Phi_Q(v;1)-1\right]+\Phi_P(u_s;1)\left[\Phi'_Q(v;1)-\Phi_Q(v;1)\right],
\end{split}
\end{align}
respectively.

Now let $\Phi_Q(1)=A$, $\Phi'_Q(1)=\overline{A}$, $\Phi_Q(v;1)=B$ and $\Phi'_Q(v;1) = \overline{B}$ (note that these quantities are constant relative to $n$ and $s$).  By Result \ref{paths}, $\Phi_P(1)= \tbinom{n +1}{2},$ $\Phi'_P(1) = \tbinom{n+2}{3},$  $\Phi_P(u_s;1) = s(n-s+1)$, and $\Phi'_P(u_s;1)=s(n-s+1)\tfrac{n+1}{2}.$  Using (\ref{TsAt1}) and (\ref{TsprimeAt1}) and substituting the values given in this paragraph, we obtain
\[
M_{T_s}=\frac{\Phi'_{T_s}(1)}{\Phi_{T_s}(1)}=\frac{\tbinom{n+2}{3}+ \overline{A}-\overline{B} + s(n-s+1)[(B-1)\tfrac{n+1}{2}+ \overline{B} -B]}{\tbinom{n+1}{2} + A-B + s(n-s+1)[B-1]}.
\]

We show that if we view $M_{T_s}$ as a real valued function of $s \in \left[1, \frac{n+1}{2}\right]$, then $M_{T_s}$ is increasing on $\left[1, \frac{n+1}{2}\right]$.  Since the denominator of $M_{T_s}$ is strictly positive on the entire interval $\left[1, \frac{n+1}{2}\right]$, the derivative of $M_{T_s}$ exists and, by the quotient rule, it has the same sign as the function $f$ defined by
\begin{align*}
f(s)&=\tfrac{d}{ds}[\Phi'_{T_s}(1)]\Phi_{T_s}(1)-\tfrac{d}{ds}[\Phi_{T_s}(1)]\Phi'_{T_s}(1).
\end{align*}
Since $\tfrac{d}{ds}(s(n-s+1)) = n-2s+1$ is a factor of $f(s)$, we see that $f(s) = 0$ when $s = \frac{n+1}{2}$.  Moreover, for $s \in \left[1, \frac{n+1}{2}\right)$, we have $(n-2s+1)>0$ so that $f(s)$ has the same sign as
\begin{align*} \tfrac{f(s)}{n-2s+1}&= \left[(B-1) \tfrac{n+1}{2}+\overline{B}-B\right]\left[\tbinom{n+1}{2} + A-B +s(n-s+1)(B-1)\right] \\
&\ \ \ \  - (B-1)\left[\tbinom{n+2}{3} + \overline{A} - \overline{B} + s(n-s+1)\left[(B-1)\tfrac{n+1}{2}+\overline{B}-B\right]\right]\\
&=\left[(B-1)\tfrac{n+1}{2}+\overline{B}-B][\tbinom{n+1}{2}+A-B\right]-(B-1)\left[\tbinom{n+2}{3} + \overline{A}-\overline{B}\right] \\
&=(B-1)\left[\tfrac{n+1}{2}\tbinom{n+1}{2} -\tbinom{n+2}{3}\right] +(\overline{B}-B)\tbinom{n+1}{2}+ (B-1)\tfrac{n+1}{2}(A-B) \\
&\ \ \ \ + (\overline{B}-B)(A-B) - (B-1)(\overline{A}-\overline{B})\\
&=(B-1)\left[\tfrac{n+1}{2}\tbinom{n+1}{2} -\tbinom{n+2}{3}\right] +(\overline{B}-B)\tbinom{n+1}{2}+ (B-1)\tfrac{n-1}{2}(A-B) \\
& \ \ \ \ +(B-1)(A-B)+ (\overline{B}-B)(A-B) - (B-1)(\overline{A}-\overline{B})\\
&=(B-1)\tfrac{n(n+1)(n-1)}{12} + (\overline{B}-B)\tbinom{n+1}{2} +(B-1)\tfrac{n-1}{2}(A-B)  \\
&\ \ \ \ +\left[(\overline{A} - A) -(\overline{B}-B)\right]+ (A\overline{B} -\overline{A}B).
\end{align*}

Note that $\tfrac{f(s)}{n-2s+1}$ does not depend on $s.$  Thus, it suffices to show that each of the terms in the final expression for $\tfrac{f(s)}{n-2s+1}$ shown above is nonnegative (and at least one is strictly positive).  Indeed, using the straightforward inequalities $B>1$, $\overline{B} > B$, $A > B$, and $n \ge 3$, it follows that
\begin{align*}
(B-1)\tfrac{n(n+1)(n-1)}{12}&>0,\\
(\overline{B}-B)\tbinom{n+1}{2}&>0, \mbox{ and}\\
(B-1)\tfrac{n-1}{2}(A-B) &>0.
\end{align*}
Let $k$ denote the order of $Q$ and assume, for $1 \le i \le k$, that $Q$ has $a_ i$ subtrees of order $i$ and $b_i$ subtrees of order $i$ that contain $v$. Then $a_i \ge b_i$ for $1 \le i \le k$ and $\displaystyle\overline{A}=\sum_{i=1}^k ia_i$, $\displaystyle A =\sum_{i=1}^k a_i$, $\displaystyle\overline{B} = \sum_{i=1}^k ib_i$ and $\displaystyle B=\sum_{i=1}^k b_i$.  Thus
\[
(\overline A - A) -(\overline{B}-B)=\sum_{i=1}^k (i-1)(a_i-b_i)\ge 0.
\]
Finally, by Result \ref{global_local},
\[
A\overline{B} -\overline{A}B=AB\left(\frac{\overline{B}}{B}-\frac{\overline{A}}{A}\right) > 0.
\]
We conclude that $f(s)$ is positive on $\left[1, \tfrac{n+1}{2}\right),$ so that $M_{T_s}$ is indeed increasing on $\left[1,\tfrac{n+1}{2}\right]$.  Returning to the discrete setting, we conclude that $M_{T_s}$ is maximized when $s=\left\lfloor \frac{n+1}{2} \right\rfloor$.
\end{proof}

Note that we have actually proven something slightly stronger than the Gluing Lemma.  Since we demonstrated that $M_{T_s}$ is increasing on the entire interval $\left[1,\tfrac{n+1}{2}\right]$, we have actually proven the following result, stated formally below as the Strong Gluing Lemma, since we refer to it later.  Essentially, with notation as in the Gluing Lemma, it says that the mean subtree order increases whenever we glue $v$ to a vertex closer to the centre of the path $P$.

\begin{lemma}[Strong Gluing Lemma]
Fix a natural number $n \ge 3$ and a  tree $Q$ of order at least $2$ having root vertex $v$. Let $P:u_1\dots u_n$ be a path of order $n$. For $s \in \{1, \dots, n\}$, let $T_s$ be the tree obtained from the disjoint union of $P$ and $Q$ by gluing $v$ to $u_s$.  If $1\leq i<j\leq \tfrac{n+1}{2}$, then $M_{T_i}<M_{T_j}.$ \hfill \qed
\end{lemma}

%\begin{corollary}\label{edge_addition_lemma}
%Fix a natural number $n \ge 3$ and a  tree $Q$ of order at least $3$ having root vertex $v$. Let $P:u_1 \ldots u_n$ be a path of order $n \ge 3$. For $s \in \{1, \ldots, n\}$, let $T_s$ be the tree obtained from the disjoint union of $P$ and $Q$ by joining $u_s$ and $v$ with an edge. Among all trees $T_s$, the optimal tree occurs if $v$ is joined to a central vertex of $P$.
%\end{corollary}
%\begin{proof}
%Let $Q'$ be obtained from $Q$ by joining, with an edge, the vertex $v$  to a new vertex $v'$. Then $T_s$ is obtained by gluing the vertex $v'$ of $Q'$ to the vertex $u_s$ of $P$. The result follows from Lemma \ref{vertex_gluing_lemma}.
%\end{proof}

The following corollary of the Gluing Lemma can be proven in a straightforward manner.

\begin{corollary}\label{diameter_n-2}
Among all trees of order $n\ge 4$ and diameter $n-2$, the tree obtained by joining a pendant vertex to a central vertex of $P_{n-1}$ is optimal. \hfill \qed
\end{corollary}

In order to state the most important corollaries of the Gluing Lemma, we require some new terminology.

\begin{definition}\label{core}
Let $T$ be a tree different from a path.
\begin{enumerate}
\item A \textit{limb} of $T$ is a maximal path in $T$ containing a leaf of $T$ and no vertices of degree greater than $2$ in $T.$
\item The tree obtained by deleting all limbs of $T$ is called the {\em core} of $T$ and is denoted by $c(T)$.
\item A limb $L$ of $T$ is \textit{adjacent} to vertex $v$ in $c(T)$ if an endnode of $L$ is adjacent to $v$ in $T.$
\item A tree $T$ is called {\em locally balanced} if for each vertex $v$ in $c(T)$, the limbs adjacent to $v$ differ in order by at most 1.
\item The number of limbs adjacent to $v$ in $c(T)$ is called the {\em limb degree} of $v$, and is denoted by $\deg_\lambda(v).$
\item The total order of the limbs adjacent to $v$ in $c(T)$ is called the {\em limb weight} at $v$, and is denoted by $w_\lambda(v)$.
\item For an ordering $\theta: v_1, v_2, \ldots, v_c$ of the vertices of $H=c(T)$, the sequence
\[
(\deg_\lambda(v_1),\deg_\lambda(v_2),\dots,\deg_\lambda(v_c))
\]
is called the {\em limb degree sequence} of $T$ relative to $\theta$ and the sequence
\[
(w_\lambda(v_1), w_\lambda(v_2),\dots, w_\lambda(v_c))
\]
is called the {\em limb weight sequence} of $T$ relative to $\theta$.
\end{enumerate}
\end{definition}

Note that if $T$ is a caterpillar, then $c(T)$ is a path.  Also if $v$ is a leaf in $c(T)$, then the limb degree of $v$ in $T$ is at least $2$.    Figure \ref{same_core_trees}  shows two trees $T_1$ and $T_2$ with the same core having the same limb degree sequences and the same limb weight sequences relative to a given vertex ordering of the core. The tree $T_1$ is not locally balanced whereas the tree $T_2$ is.  Note that a locally balanced tree with a given core, limb degree sequence, and limb weight sequence is unique up to isomorphism.  The next corollary states that the locally balanced tree is optimal among all trees with the same core, limb weight sequence, and limb degree sequence.

\begin{figure} [h!]
\begin{center} \label{same_core_trees}
\includegraphics[scale=0.65]{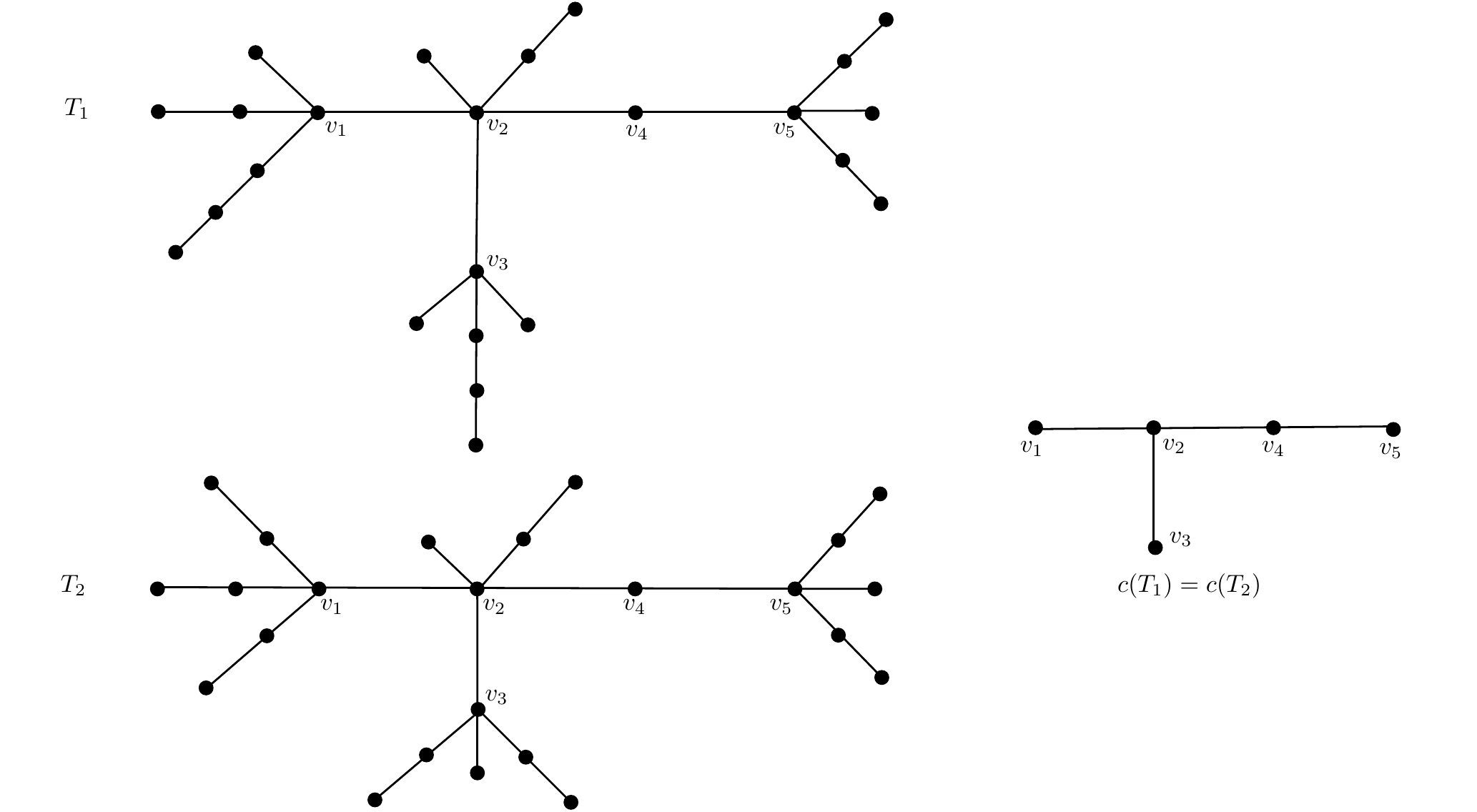}
\caption{Trees with the same core, the same limb degree sequence $(3,2,3,0,3)$, and the same limb weight sequence $(6,3,5,0,5)$, both relative to $\theta\colon\ v_1,v_2,v_3,v_4,v_5$.}
\end{center}
\end{figure}

\begin{theorem}\label{same_core}
Let $H$ be a tree with a vertex ordering $\theta$, and fix a limb weight sequence $\Lambda$ and a limb degree sequence $\Delta$ relative to $\theta$.  Among all trees with core $H$, limb weight sequence $\Lambda$, and limb degree sequence $\Delta$, the locally balanced tree is optimal.
\end{theorem}
\begin{proof}
Let $T$ be optimal among all trees with core $H$, limb weight sequence $\Lambda$, and limb degree sequence $\Delta$, and suppose that $T$ is not locally balanced.  Then there is some vertex $v$ in $c(T)$ such that two limbs adjacent to $v$, say $P_1$ and $P_2$, differ in order by at least $2$.  Assume $n(P_2) \ge n(P_1) +2$ and let $P$ be the path obtained by joining a new vertex $u$ to an end vertex of $P_1$ and an end vertex of $P_2$ and let $Q$ be obtained by deleting the vertices of $P_1 \cup P_2$ from $T$. Then $T$ is obtained by gluing the vertex $u$ of $P$ with the vertex $v$ of $Q$.  By our assumption about $P_1$ and $P_2$ and the Gluing Lemma, the tree $\overline{T}$ obtained from $P$ and $Q$ by gluing $v$ to a central vertex of $P$ has mean subtree order that exceeds that of $T$. This contradicts our choice of $T$ since $\overline{T}$ still has core $H$, limb weight sequence $\Lambda,$ and limb degree sequence $\Delta.$
\end{proof}

Theorem \ref{same_core} applies to several families of particular interest to us.  The following results follow immediately from Theorem \ref{same_core}.

\begin{corollary} \label{asters}
Among all asters of fixed order and with a fixed number of leaves, the balanced aster is optimal.\hfill \qed
\end{corollary}

\begin{corollary} \label{stickmen}
Among all stickmen of a fixed order obtained from paths $P$ and $Q$ by joining an internal vertex of $P$ with an internal vertex of $Q$ by a path of fixed order, the locally balanced stickman is optimal. \hfill \qed
\end{corollary}

We have seen that the optimal tree among all trees with the same core, limb weight sequence, and limb degree sequence is the locally balanced tree.  The next corollary describes the optimal tree when we remove the restriction on the limb degree sequence.

\begin{theorem}\label{same_core_and_limb_weight}
Let  $H$ be a tree with vertex ordering $\theta$ and fix a limb weight sequence $\Lambda$ relative to $\theta$.  Among all trees with core $H$ and limb weight sequence $\Lambda$, the optimal tree is precisely the one whose limbs all have order $1.$
\end{theorem}
\begin{proof} Let $T$ be an optimal tree among all trees with core $H$ and limb weight sequence $\Lambda$.  Suppose that some limb $L$ adjacent to vertex $v$ of $H$ has order at least $2$.  Let $L'$ be obtained from $L$ by joining a new vertex $w$ to an endnode $u$ of $L$.  Let $T'$ be obtained by deleting $L$ from $T$ and then gluing $v$ to the vertex $u$ in $L'$. Then $T'$ has core $H$ and limb weight sequence $\Lambda,$ and by the Strong Gluing Lemma, the mean subtree order of $T'$ exceeds that of $T$, a contradiction.
\end{proof}

The optimal tree among all trees with the same core and limb weight sequence as the trees $T_1$ and $T_2$ of Figure \ref{same_core_trees} is shown in Figure \ref {optimal}.

\begin{figure} [h!]
\begin{center}
\includegraphics[scale=0.65]{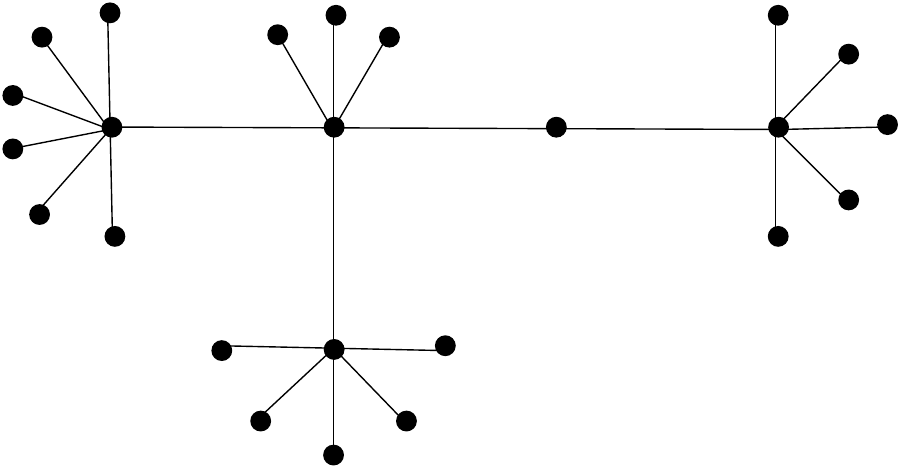}
\caption{The optimal tree for a given core and limb weight sequence.}\label{optimal}
\end{center}
\end{figure}

It follows that between the stickman and the subdivided double star with the same interior path and the same limb weight at each end of the path, the subdivided double star has higher mean subtree order.  This fact was alluded to in \cite{j1} but was not proven there.  The following necessary condition on any optimal tree in $\T{n}$ also follows immediately from Theorem~\ref{same_core_and_limb_weight}.

\begin{theorem}\label{ShortLimbs}
If $T$ is optimal in $\T{n}$ for some $n\geq 4,$ then the limbs of $T$ all have order $1$, i.e.\ every leaf of $T$ is adjacent to a vertex of degree at least $3$ in $T.$ \hfill \qed
\end{theorem}

Finally, we show that the Gluing Lemma gives us a positive answer to open problem (7.6) from \cite{j1}. Let $T$ be a tree that is not a path. Let $w$ be a leaf of $T$ and let $v$ be the vertex of degree at least $3$ in $T$ that is closest to $w$ (i.e.\ $v$ is adjacent to the limb $L_w$ of $T$ containing $w$).  Let $u$ be a neighbour of $v$ that is not on the shortest $w$--$v$ path (i.e.\ not on the limb $L_w$).  Then the tree $T'$ obtained from $T$ by deleting $vu$ and adding $uw$ is called a {\em standard} $1$-{\em associate} of $T$. Jamison~\cite{j1} showed that if $T$ is a tree of order $n$ with mean subtree order less than $(n+1)/2$, then there is a standard $1$-associate of $T$ whose mean subtree order is less than that of $T$.  Moreover, he conjectured that every tree not isomorphic to a path has a standard $1$-associate of lower mean subtree order.  We prove this statement below.

\begin{theorem}\label{associate}
If $T$ is a tree that is not a path, then $T$ has a standard $1$-associate $T'$ such that $M_T>M_{T'}$.
\end{theorem}

\begin{proof}
Recall that any leaf of the core of $T$ has limb degree at least $2.$  Let $v$ be a leaf in $c(T)$, let $L_1$ and $L_2$ be two limbs adjacent to $v$, and let $Q$ be the tree obtained from $T$ by deleting $L_1$ and $L_2.$  Let $L$ be the path obtained from the disjoint union of $L_1$ and $L_2$ by adding a new vertex joined to a leaf of $L_1$ and a leaf of $L_2$.  Let $T'$ be the tree obtained from the disjoint union of $L$ and $Q$ by gluing the vertex $v$ of $Q$ to a leaf of $L$. Then $T'$ is a standard $1$-associate of $T$ since it can be obtained by deleting the edge between $v$ and its neighbour $w$ in $L_2$ and adding an edge between $w$ and the leaf of $T$ that is in $L_1$.   By the Strong Gluing Lemma, $T'$ has smaller mean subtree order than $T$.
\end{proof}

It follows immediately from Theorem~\ref{associate} that the path has minimum mean subtree order in $\T{n}.$  Thus, the Gluing Lemma leads to a conceptually simple alternate proof of the main result of~\cite{j1}, stated formally below.

\begin{theorem}\label{PathMinimum}
If $T$ is a tree of order $n$, then $M_T\geq \tfrac{n+2}{3}$ with equality if and only if $T$ is a path.  \hfill \qed
\end{theorem}

We conclude that the Gluing Lemma is a tool of great strength and utility.  Though the proof we give is long and involved, we have shown that is has numerous significant (and surprising) implications for the mean subtree order problem.

\section{Optimal subdivided double stars}\label{DoubleStarSection}

Motivated by Jamison's observation that batons can have density arbitrarily close to $1$ and the fact that the optimal tree in $\T{n}$ is rather baton-like for all $9\leq n\leq 24$, we undertake an in-depth study of the mean subtree order of general subdivided double stars in this section.  We first demonstrate that among all subdivided double stars of a fixed order and a fixed even number of leaves, the baton (the \textit{balanced} subdivided double star) is optimal, as long as the number of leaves is sufficiently large.  Then we determine the asymptotic growth of the number of leaves in an optimal baton of a fixed order.

Before we begin with our results, we discuss the subtrees of subdivided double stars in general.  Consider $D_n(s,2m-s),$ the subdivided double star on $n\geq 2m+2$ vertices with $2m$ leaves in total; $s$ at one end and $2m-s$ at the other.  Note that the interior path of $D_n(s,2m-s)$ has order $n-2m-2.$  The subtrees of $D_n(s,2m-s)$ can be partitioned into three groups:
\begin{itemize}
\item Those that do not contain the centre vertex of either star.  There are $2m+\binom{n-2m-1}{2}$ such subtrees and the sum of their orders is $2m+\binom{n-2m}{3}.$
\item Those that contain the centre vertices of both stars.  There are $2^{2m}$ such subtrees and the sum of their orders is $(n-m)\cdot 2^{2m}.$
\item Those that contain the centre of exactly one of the stars.  There are
\[
(n-2m-1)\cdot (2^s+2^{2m-s})
\]
such subtrees and the sum of their orders is
\[
\tfrac{1}{2}(n-2m-1)\left[(n-2m+s)\cdot 2^s+(n-s)\cdot 2^{2m-s}\right]
\]
\end{itemize}

In the next theorem, we consider the family of subdivided double stars with a fixed number of vertices and a fixed \emph{even} number of leaves.  We do this simply because every baton has an even number of leaves, and they are our primary interest in this section.  Based on computational evidence, we suspect that a similar result holds for the family of subdivided double stars on a fixed number of vertices and a fixed \emph{odd} (and sufficiently large) number of leaves; that is, we suspect that the ``nearly balanced'' subdivided double star is optimal in this family.

\begin{theorem}\label{BalancedDoubleStars}
Let $m,n\in\mathbb{N}$ with $2m\leq n-2.$  Among all subdivided double stars on $n$ vertices with $2m$ leaves, the balanced subdivided double star (i.e.\ the baton) $D_n(m,m)$ is optimal whenever $m\geq \log_2(n).$
\end{theorem}

\begin{proof}
Let $n$ and $m$ be as in the theorem statement.  We wish to show that $M_{D_n(m,m)}-M_{D_n(s,2m-s)}>0$ for $1\leq s\leq m-1.$  It suffices to show that the difference
\begin{align}\label{h0}
h_{n,m}(s)=\Phi'_{D_n(m,m)}(1)\cdot \Phi_{D_n(s,2m-s)}(1)-\Phi'_{D_n(s,2m-s)}(1)\cdot \Phi_{D_n(m,m)}(1)
\end{align}
is positive for $1\leq s\leq m-1.$  Using a computer algebra system, we have verified this statement for all possible cases with $m\leq 4$ (note that we only need to check up to $n=16$ since $n>16$ implies $\log_2(n)>4$), so we may assume that $m\geq 5.$

We note that the number and total order of those subtrees of $D_n(s,2m-s)$ of the first two types listed in the discussion preceding the theorem statement do not depend on $s$.  We let $A_{n,m}$ and $\overline{A}_{n,m}$ denote the number and total order of these subtrees, respectively; that is
\begin{equation*}
\begin{split}
A_{n,m}&=2m+\tbinom{n-2m-1}{2}+2^{2m}, \mbox{ and}\\
\overline{A}_{n,m}&=2m+\tbinom{n-2m}{3}+(n-m)\cdot 2^{2m}.
\end{split}
\end{equation*}
The number and total order of the remaining subtrees (those that contain the centre vertex of exactly one star) do depend on $s.$  We let $B_{n,m}(s)$ and $\overline{B}_{n,m}(s)$ denote the number and total order of these subtrees, respectively; that is
\begin{equation*}
\begin{split}
B_{n,m}(s)&=(n-2m-1)\cdot (2^s+2^{2m-s}), \mbox{ and}\\
\overline{B}_{n,m}(s)&=\tfrac{1}{2}(n-2m-1)\left[(n-2m+s)\cdot 2^s+(n-s)\cdot 2^{2m-s}\right].
\end{split}
\end{equation*}
With this notation, we have
\begin{equation}\label{AplusB}
\begin{split}
\Phi_{D_n(s,2m-s)}(1)&=A_{n,m}+B_{n,m}(s), \mbox{ and}\\
\Phi'_{D_n(s,2m-s)}(1)&=\overline{A}_{n,m}+\overline{B}_{n,m}(s).
\end{split}
\end{equation}

Substituting the expressions of (\ref{AplusB}) into (\ref{h0}) and then expanding and regrouping, we obtain
\begin{equation}
\begin{split}\label{h1}
h_{n,m}(s)&=\overline{A}_{n,m}\left[B_{n,m}(s)-B_{n,m}(m)\right]-A_{n,m}\left[\overline{B}_{n,m}(s)-\overline{B}_{n,m}(m)\right]\\
& \ \ \ \ -\left[\overline{B}_{n,m}(s)\cdot B_{n,m}(m)-\overline{B}_{n,m}(m)\cdot B_{n,m}(s)\right].
\end{split}
\end{equation}
By a series of long calculations, we show that $h_{n,s}(s)>0$ for $1\leq s\leq m-1$, which completes the proof.  In order to save space and ease readability, the details are given in Appendix A.
\end{proof}

We have shown that the baton is optimal among all subdivided double stars of a fixed order and a fixed even number $2m$ of leaves, whenever $m\geq \log_2 n$.  Computations show that the same result does not necessarily hold when $m<\log_2 n$.

A natural next step is to determine the structure of the optimal tree(s) among all batons of a fixed order.  Our next result describes the asymptotic structure of any such optimal baton.

\begin{theorem}\label{BatonAsymptotic}
Let $s_n$ be a number such that $D_n(s_n,s_n)$ is optimal among all batons of order $n$.  Then for $n$ sufficiently large,
\[
2\log_2(n)-2<s_n<2\log_2(n)+1.
\]
\end{theorem}

\begin{proof}
Let $s_n$ be as in the theorem statement and let $s\in\mathbb{N}$ with $s\leq \tfrac{n-2}{2}.$  We consider the difference $M_{D_n(s+1,s+1)}-M_{D_n(s,s)}$, which has the same sign as
\begin{align}\label{Difference}
f_n(s)=\Phi'_{D_n(s+1,s+1)}(1)\cdot \Phi_{D_n(s,s)}(1)-\Phi'_{D_n(s,s)}(1)\cdot \Phi_{D_n(s+1,s+1)}(1).
\end{align}
From the discussion preceding Theorem \ref{BalancedDoubleStars},
\begin{align}\label{BatonNumber}
\Phi_{D_n(s,s)}(1)=2s+\binom{n-2s-1}{2}+2(n-2s-1)2^s+2^{2s}
\end{align}
and
\begin{align}\label{BatonTotal}
\Phi'_{D_n(s,s)}(1)=2s+\binom{n-2s}{3}+(n-2s-1)(n-s)2^s+(n-s)2^{2s},
\end{align}
and the analogous expressions for $D_n(s+1,s+1)$ are obtained by replacing $s$ with $s+1$.  At this point we used a computer algebra system to substitute these expressions for $\Phi_{D_n(s,s)}(1),$ $\Phi'_{D_n(s,s)}(1),$ $\Phi_{D_n(s+1,s+1)}(1),$ and $\Phi'_{D_n(s+1,s+1)}(1)$ into (\ref{Difference}), and then expand and collect terms.  This resulted in an expression for $f_n(s)$ as a fourth degree polynomial in $n$ where the coefficients are functions of $s.$  Explicitly, we have
\begin{align*}
f_n(s)=c_4(s)n^4+c_3(s)n^3+c_2(s)n^2-c_1(s)n-c_0(s),
\end{align*}
where
\begin{align*}
c_4(s)&=\tfrac{1}{6}\cdot 2^s-\tfrac{1}{6},\\
c_3(s)&=2^{2s}-\tfrac{1}{6}(5s+14)\cdot 2^s+\tfrac{4}{3}(s+1),\\
c_2(s)&=2\cdot 2^{3s}-\tfrac{1}{2}(9s+16)\cdot 2^{2s}\\
&\ \ \ \ +\tfrac{1}{6}(6s^2+84s+59)\cdot 2^s-\tfrac{1}{6}(24s^2+60s+23),\\
c_1(s)&=(6s+8)\cdot 2^{3s}-\tfrac{1}{2}(12s^2+71s+42)\cdot 2^{2s}\\
& \ \ \ \ -\tfrac{1}{6}(4s^3-156s^2-347s-106)\cdot 2^s\\
& \ \ \ \ -\tfrac{1}{3}(16s^3+72s^2+82s+14), \mbox{ and}\\
c_0(s)&=4\cdot 2^{4s}-(4s^2+18s+14)\cdot 2^{3s}+(2s^3+33s^2+60s+18)\cdot 2^{2s}\\
& \ \ \ \ +\tfrac{1}{3}(4s^4-44s^3-199s^2-208s-30)\cdot 2^s\\
& \ \ \ \ +\tfrac{1}{3}(8s^4+56s^3+118s^2+82s+6).
\end{align*}

We first claim that each of these coefficient functions is bounded above by its leading term for all $s\geq 1$; that is,
\begin{align}
c_4(s)&< \tfrac{1}{6}\cdot 2^s,\label{c4upper}\\
c_3(s)&< 2^{2s},\label{c3upper}\\
c_2(s)&< 2\cdot 2^{3s},\label{c2upper}\\
c_1(s)&< (6s+8)\cdot 2^{3s},\mbox{ and}\label{c1upper}\\
c_0(s)&< 4\cdot 2^{4s}\label{c0upper}
\end{align}
for $s\geq 1.$  The proof of each inequality involves straightforward (but at times quite tedious) grouping of the non-leading terms and simple inequalities.  This work is shown in Appendix B.  Further, we have
\begin{align}
c_4(s)&>0 \mbox{ for } s\geq 1,\label{c4lower}\\
c_3(s)&>0 \mbox{ for } s\geq 2,\label{c3lower}\\
c_2(s)&> \tfrac{5}{3}\cdot 2^{3s} \mbox{ for } s\geq 7,\label{c2lower}\\
c_1(s)&>0 \mbox{ for } s\geq 5, \mbox{ and}\label{c1lower}\\
c_0(s)&>\tfrac{19}{6}\cdot 2^{4s} \mbox{ for } s\geq 10,\label{c0lower}
\end{align}
by similar work, also shown in Appendix B.

Now suppose that $s\geq 2\log_2(n)$ and $s\geq 10$ (this second condition follows immediately from the first when $n\geq 32$).  Equivalently, we have $2^s\geq n^2$ and $s\geq 10$.  By (\ref{c1lower}) and (\ref{c0lower}), and then (\ref{c4upper}), (\ref{c3upper}), and (\ref{c2upper}), we have
\begin{align*}
c_0(s)+c_1(s)n> \tfrac{19}{6}\cdot 2^{4s}&=\tfrac{1}{6}\cdot 2^{4s}+2^{4s}+2\cdot 2^{4s}\\
&\geq \tfrac{1}{6}\cdot 2^sn^6+2^{2s}n^4+2\cdot 2^{3s}n^2\\
&>c_4(s)n^4+c_3(s)n^3+c_2(s)n^2,
\end{align*}
and it follows that $f_n(s)$ is negative.  This means that when $n\geq 32$ and $s\geq 2\log_2(n)$ we have $M_{D_n(s+1,s+1)}<M_{D_n(s,s)}$.  So for $n\geq 32$ we have $s_n<2\log_2(n)+1.$

On the other hand, suppose that $s\leq 2\log_2(n)-2$, which is equivalent to $n^2\geq 4\cdot 2^s$.  It follows that $s\leq \tfrac{n-12}{9}$ (or equivalently $n\geq 9s+12$) for $n\geq 120$.  Thus, if $s\leq 2\log_2(n)-2$, $n\geq 120$ and $s\geq 7$, then by (\ref{c4lower}), (\ref{c3lower}), and (\ref{c2lower}), and then (\ref{c1upper}) and (\ref{c0upper}), we have
\begin{align*}
c_4(s)n^4+c_3(s)n^3+c_2(s)n^2> \tfrac{5}{3}\cdot 2^{3s}n^2&=2^{3s}n^2+\tfrac{2}{3}\cdot 2^{3s}n^2\\
&\geq 4\cdot 2^{4s}+(6s+8)2^{3s}n\\
&> c_0(s)+c_1(s)n.
\end{align*}
Hence, $f_n(s)$ is positive in this case.  Finally, for the remaining cases $s\leq 6,$ we can verify directly that $f_n(s)$ is positive for $n$ sufficiently large (for each case $s=1,2,\dots,6$ we get a quartic in $n$ with positive leading coefficient).  In fact, we find that $f_n(s)>0$ for all $s\leq 6$ whenever $n\geq 20.$  We conclude that for $n\geq 120$, if $s\leq 2\log_2(n)-2,$ then $M_{D_n(s,s)}<M_{D_n(s+1,s+1)}$, so that $s_n>2\log_2(n)-2.$
\end{proof}

We glean from Theorem \ref{BatonAsymptotic} that the baton $D_n(\lceil 2\log_2(n)\rceil,\lceil 2\log_2(n)\rceil)$ is likely close to optimal among all batons of order $n$.  The bulk of the proof of the next result involves giving a lower bound on the mean subtree order of this tree.

\begin{corollary}\label{MaxMeanCorollary}
For each natural number $n$, there is a caterpillar $C_n$ of order $n$ satisfying $M_{C_n}>n-\lceil 2\log_2(n)\rceil-1$.
\end{corollary}

\begin{proof}
For ease of reading, let $s_n=\lceil 2\log_2(n)\rceil.$  It is easily verified that the mean subtree order of the star $K_{1,n-1}$ is strictly greater than $\tfrac{n}{2}$ for all $n\in\mathbb{N},$ so we may assume that $n-s_n-1\geq \tfrac{n}{2},$ or equivalently $n\geq 2s_n+2.$  We claim that the mean subtree order of the baton $D_n(s_n,s_n)$ is greater than $n-s_n-1$ for $n\geq 2s_n+2$ (note that the baton $D_n(s_n,s_n)$ is well-defined in this case).  It suffices to show that the difference
\begin{align}\label{MeanBound1}
p(n)=\Phi'_{D_n(s_n,s_n)}(1)-(n-s_n-1)\Phi_{D_n(s_n,s_n)}(1)
\end{align}
is positive for $n\geq 2s_n+2.$  We evaluate (\ref{BatonNumber}) and (\ref{BatonTotal}) at $s=s_n$ to obtain expressions for $\Phi_{D_n(s_n,s_n)}(1)$ and $\Phi'_{D_n(s_n,s_n)}(1)$, respectively, and then substitute these expressions into (\ref{MeanBound1}) to obtain
\begin{align*}
p(n)&=2^{2s_n}+\left[(n-2s_n-1)(n-s_n)-2(n-s_n-1)(n-2s_n-1)\right]2^{s_n}\\
&\ \ \ \ +\tbinom{n-2s_n}{3}-(n-s_n-1)\tbinom{n-2s_n-1}{2}-2s_n(n-s_n-2).
\end{align*}
We then expand the expression inside the square brackets above and apply rather rough inequalities to the terms without exponential factors (including simply dropping the positive term) to obtain
\begin{align*}
p(n)&> 2^{2s_n}+\left[-n^2+3ns_n+3n-(2s_n^2+5s_n+2)\right]2^{s_n}-\tfrac{n^3}{2}-n^2\\
&= \left[2^{s_n}-n^2+3ns_n+3n-(2s_n^2+5s_n+2)\right]2^{s_n} -\tfrac{n^3}{2}-n^2.
\end{align*}
Consider the expression inside the square brackets above.  From the fact that $s_n=\lceil 2\log_2 n\rceil,$ we have $2^{s_n}\geq n^2$, so
\begin{align*}
2^{s_n}-n^2+3ns_n+3n-(2s_n^2+5s_n+2)&\geq 3ns_n+3n-(2s_n^2+5s_n+2)\\
&\geq 6s_n^2+6s_n-(2s_n^2+5s_n+2)+3n\\
&\geq 3n,
\end{align*}
where the second inequality follows from the assumption that $n\geq 2s_n+2$.  Thus
\[
p(n)>3n\cdot 2^{s_n}-\tfrac{n^3}{2}-n^2\geq 3n^3-\tfrac{n^3}{2}-n^2>0.
\]
This completes the proof.
\end{proof}

Together with Jamison's upper bound on the mean subtree order of $T$ in terms of the number of leaves of $T$ (Result \ref{LeavesLemma}), Corollary \ref{MaxMeanCorollary} tells us that the number of leaves in an optimal tree in $\T{n}$ grows at most logarithmically in $n.$  This necessary condition for optimality is stated formally below, and applies equally well to the family $\C{n}$ of all caterpillars of order $n$.

\begin{corollary}\label{LeavesCorollary}
Let $T$ be a tree of order $n$ with $\ell$ leaves.  If $T$ is optimal in $\T{n}$ (or $\C{n}$), then $\ell< 2\lceil2\log_2 (n)\rceil+2.$
\end{corollary}

\begin{proof}
Suppose that $T$ has $\ell\geq 2\lceil 2\log_2 (n)\rceil +2$ leaves.  By Result \ref{LeavesLemma},
\[
M_{T}\leq n-\frac{2\lceil 2\log_2(n)\rceil-2}{2}=n-\lceil 2\log_2(n)\rceil-1.
\]
This is a contradiction since, by Corollary \ref{MaxMeanCorollary}, there is a tree in $\C{n}\subseteq \T{n}$ with mean subtree order greater than $n-\lceil 2\log_2 n\rceil-1.$
\end{proof}

By Theorem \ref{ShortLimbs}, every leaf of $T$ is adjacent to a vertex of degree at least $3$ in $T$.  Hence, if $T$ is optimal in $\T{n}$, then the number of leaves of $T$ is at least twice the number of twigs of $T.$  Therefore, by Corollary \ref{LeavesCorollary}, if $T$ is optimal in $\T{n}$ and has $t$ twigs, then $t\leq \lceil 2\log_2 n\rceil+1.$   We can do slightly better using Result \ref{HasTwigs} instead.

\begin{corollary}\label{TwigsCorollary}
Let $T$ be a tree of order $n$ with $t$ twigs.  If $T$ is optimal in $\T{n}$, then $t<\tfrac{5}{7}\lceil2\log_2 n\rceil+\tfrac{5}{7}$.
\end{corollary}

\begin{proof}
Suppose that $T$ is optimal in $\T{n}$ and $t\geq \tfrac{5}{7}\lceil\log_2 n\rceil +\tfrac{5}{7}.$  Then by Theorem \ref{ShortLimbs}, every leaf of $T$ must be adjacent to a vertex of degree at least $3$ in $T$.  In particular, every twig of $T$ has degree at least $3.$  Thus, by Result \ref{HasTwigs},
\[
M_T\leq n-\tfrac{7}{5}t\leq n-\lceil 2\log_2 n\rceil-1.
\]
This is a contradiction since, by Corollary \ref{MaxMeanCorollary}, there is a tree in $\T{n}$ with mean subtree order greater than $n-\lceil 2\log_2 n\rceil-1.$
\end{proof}

\section{Optimal bridges}\label{BridgeSection}

In this section, we describe the asymptotic structure of the optimal tree(s) among all bridges of a fixed order.  We contrast the total limb weight of these optimal bridges with the total limb weight of the optimal batons of the same order, and demonstrate that the mean subtree order for the optimal bridges is indeed much lower than the mean subtree order for the optimal batons.

Let $s\geq 1$, $t\geq 0$, and let $u$ and $v$ be the vertices of degree $3$ in $B(s,t)$.  The subtrees of $B(s,t)$ can be partitioned into three types:
\begin{itemize}
\item Those that contain neither $u$ nor $v$.  There are $4\binom{s+1}{2}+\binom{t+1}{2}$ such subtrees, and the sum of their orders is $4\binom{s+2}{3}+\binom{t+2}{3}$.
\item Those that contain $u$ or $v$ but not both.  There are $2(s+1)^2(t+1)$ such subtrees and they have mean order $(2s+t+2)/2$, so the sum of their orders is $(s+1)^2(t+1)(2s+t+2)$.
\item Those that contain both $u$ and $v.$  There are $(s+1)^4$ such subtrees and they have mean order $(2s+t+2)$, so the sum of their orders is $(2s+t+2)(s+1)^4$.
\end{itemize}
So the number of subtrees of $B(s,t)$ is given by
\begin{align}\label{BridgeNumber}
\Phi_{B(s,t)}(1)=4\binom{s+1}{2}+\binom{t+1}{2}+2(s+1)^2(t+1)+(s+1)^4,
\end{align}
and the total number of vertices contained in these subtrees is given by
\begin{align}\label{BridgeTotal}
\Phi'_{B(s,t)}(1)=4\binom{s+2}{3}+\binom{t+2}{3}+(2s+t+2)(s+1)^2\left[(t+1)+(s+1)^2\right].
\end{align}

Now we focus on the bridges of fixed order $n+2.$  Take special note of the fact that $n$ does not stand for the order of the tree here -- letting the order be $n+2$ instead makes the following theorem and its proof significantly simpler to write down.  For ease of notation, we let $B_n(s)=B(s,n-4s)$ and we let $\mathcal{B}_n$ denote the set $\left\{B_n(s)\colon\ s\in\mathbb{N},s\leq \tfrac{n}{4}\right\}$ of all bridges of order $n+2.$

\begin{theorem}\label{AsymptoticBridge}
Fix a real number $k>1.$  For $n\in\mathbb{N}$, let $s_n$ be a number such that $B_n(s_n)$ is optimal in $\mathcal{B}_n$.  Then for $n$ sufficiently large (depending on $k$),
\[
\tfrac{n^{2/3}}{k}< s_n< n^{2/3}.
\]
In particular, $s_n$ grows asymptotically like $n^{2/3}$.
\end{theorem}

\begin{proof}
We first demonstrate the asymptotic lower bound on $s_n$.  Let $s\leq \tfrac{n}{4}$ and consider the difference $M_{B_n(s+1)}-M_{B_n(s)}$, which has the same sign as
\begin{align}\label{BridgeDifference}
g_n(s)=\Phi'_{B_n(s+1)}(1)\cdot \Phi_{B_n(s)}(1)-\Phi'_{B_n(s)}(1)\cdot \Phi_{B_n(s+1)}(1).
\end{align}
At this point we used a computer algebra system to evaluate (\ref{BridgeNumber}) and (\ref{BridgeTotal}) at the appropriate values of $s$ and $t$, substitute these expressions into (\ref{BridgeDifference}), and finally expand and collect terms.  This resulted in an expression for $g_n(s)$ as a fourth degree polynomial in $n$ where the coefficients are polynomials in $s.$  Explicitly,
\begin{align*}
g_n(s)=c_4(s)n^4+c_3(s)n^3+c_2(s)n^2-c_1(s)n-c_0(s),
\end{align*}
where
\begin{align*}
c_4(s)&=\tfrac{1}{3}s+\tfrac{1}{6},\\
c_3(s)&=\tfrac{4}{3}s^3+s^2+\tfrac{11}{3}s+\tfrac{5}{3},\\
c_2(s)&=2s^5+2s^3-6s^2+\tfrac{17}{3}s+\tfrac{23}{6},\\
c_1(s)&=14s^6+40s^5+66s^4+\tfrac{160}{3}s^3+57s^2+\tfrac{59}{3}s+\tfrac{5}{3}, \mbox{ and}\\
c_0(s)&=2s^8-8s^7-32s^6-28s^5+32s^4+64s^3+76s^2+38s+8.
\end{align*}

We would like to show that for $n$ sufficiently large, $g_n(s)>0$ whenever $s\leq\tfrac{ n^{2/3}}{k}.$  We make use of the following inequalities, which hold for the values of $s$ indicated:
\begin{alignat*}{2}
c_3(s)&> \tfrac{4}{3}s^3  \hspace{0.5cm}& \mbox{ for } s&\geq 1, \\
c_2(s)&> 2s^5  &\mbox{ for } s&\geq  1,\\
c_1(s)&< 18s^6  &\mbox{ for } s&\geq 12, \mbox{ and}\\
c_0(s)&< 2s^8   &\mbox{ for } s&\geq 2.
\end{alignat*}

Note that if we fix a natural number $s$, then $g_n(s)$ is a quartic in $n$ with real coefficients, and the coefficient $c_4(s)=\tfrac{1}{3}s+\tfrac{1}{6}$ of the leading term is strictly positive.  Thus, for any fixed $s$ we will have $g_n(s)>0$ for $n$ sufficiently large, say $n\geq n_s$.  Further, if $r>0$ is any fixed real number, then for $n$ sufficiently large we will have $g_n(s)>0$ for all $s< r$ (by taking $n\geq \max\{n_s\colon\ s<r\}$).

Define constant $r_k$ by
\[
r_k=\max\left\{12,\left(\frac{18-\tfrac{4}{3}k^3}{2\left(1-\tfrac{1}{k^3}\right)k^{3/2}}\right)^2\right\}.
\]
By the argument of the preceding paragraph, if $s< r_k$ then for $n$ sufficiently large, $g_n(s)>0$.  So we may now assume that $s\geq r_k.$  Note that $r_k\geq 12$, so the inequalities on the coefficient functions given above all hold.

Now, for $s\leq \tfrac{n^{2/3}}{k}$ (which is equivalent to $n^2\geq k^3s^3$), we have
\begin{align*}
c_2(s)n^2+c_3(s)n^3&> 2s^5n^2+\tfrac{4}{3}s^3n^3\\
&=\tfrac{2}{k^3}s^5n^2+2\left(1-\tfrac{1}{k^3}\right)s^5n^2+\tfrac{4}{3}s^3n^3\\
&\geq\tfrac{2}{k^3}k^3s^8+2\left(1-\tfrac{1}{k^3}\right)k^{3/2}s^{13/2}n+\tfrac{4}{3}k^3 s^6n\\
&=2s^8+\left[2\left(1-\tfrac{1}{k^3}\right)k^{3/2}\sqrt{s}+\tfrac{4}{3}k^3\right]s^6n\\
&\geq 2s^8+18s^6n\\
&> c_0(s)+c_1(s)n.
\end{align*}
Note that the inequality $2\left(1-\tfrac{1}{k^3}\right)k^{3/2}\sqrt{s}+\tfrac{4}{3}k^3\geq 18$ follows immediately from the assumption that $s\geq r_k$ (and this was the motivation for the definition of $r_k$).  Thus, for $n$ sufficiently large, we conclude that if $s\leq \frac{n^{2/3}}{k}$, then $g_n(s)>0,$ meaning that $M_{B_n(s+1)}>M_{B_n(s)}.$  Therefore, we must have $s_n>\tfrac{n^{2/3}}{k}$ for $n$ sufficiently large.\\

Now we demonstrate the upper bound $s_n< n^{2/3}.$  Unlike the lower bound on $s_n$, this upper bound holds for all $n.$  Note that it is trivially true when $n<64$ since then $n^{2/3}>\tfrac{n}{4}$ (and $s\leq \tfrac{n}{4}$ in general).  To prove the bound for $n\geq 64$, we will show that $M_{B_n(s-1)}>M_{B_n(s)}$ when $n\geq 64$ and $s\geq n^{2/3}$.  We expand and simplify the difference
\begin{align*}\label{BridgeDifference2}
h_n(s)=\Phi'_{B_n(s-1)}(1)\cdot \Phi_{B_n(s)}(1)-\Phi'_{B_n(s)}(1)\cdot \Phi_{B_n(s-1)}(1)
\end{align*}
using a computer algebra system to obtain
\[
h_n(s)=c_0(s)+c_1(s)n-c_2(s)n^2-c_3(s)n^3-c_4(s)n^4,
\]
where
\begin{align*}
c_0(s)&=2s^8-24s^7+80s^6-116s^5+112s^4-96s^3+100s^2-70s+20,\\
c_1(s)&=14s^6-44s^5+76s^4-\tfrac{272}{3}s^3+103s^2-\tfrac{247}{3}s+\tfrac{77}{3}, \\
c_2(s)&=2s^5-10s^4+22s^3-32s^2+\tfrac{101}{3}s-\tfrac{71}{6}, \\
c_3(s)&=\tfrac{4}{3}s^3-3s^2+\tfrac{17}{3}s-\tfrac{7}{3},\mbox{ and} \\
c_4(s)&=\tfrac{1}{3}s-\tfrac{1}{6}.
\end{align*}
We make use of the following inequalities, which hold for the values of $s$ indicated:
\begin{alignat*}{2}
c_0(s)&> 2s^8-24s^7  \hspace{0.5cm}& \mbox{ for } s&\geq 1, \\
c_1(s)&> 8s^6  &\mbox{ for } s&\geq 6,\\
c_2(s)&< 2s^5 &\mbox{ for } s&\geq 2,\\
c_3(s)&< \tfrac{4}{3}s^3&\mbox{ for } s&\geq 2,  \mbox{ and}\\
c_4(s)&< \tfrac{1}{3}s & \mbox{ for } s&\geq 1.
\end{alignat*}

Now if $s\geq n^{2/3}$, then the assumption $n\geq 64$ gives $s\geq 16.$  So all of the above inequalities hold.  Using these inequalities along with the assumption $s\geq n^{2/3}$, we have
\begin{alignat*}{2}
c_0(s)+c_1(s)n&> 2s^8-24s^7+8s^6n\\
&>2s^8-24s^7+6s^6n+\tfrac{4}{3}s^6n+\tfrac{1}{3}s^6n \ \ &\\
&=2s^8+6s^6(n-4s)+\tfrac{4}{3}s^6n+\tfrac{1}{3}s^6n&\\
&\geq 2s^8+\tfrac{4}{3}s^6n+\tfrac{1}{3}s^6n \ \ \ &\mbox{(since $s\leq \tfrac{n}{4}$)}\\
&>2s^5n^2+\tfrac{4}{3}s^3n^3+\tfrac{1}{3}sn^{13/3}&\\
&>c_2(s)n^2+c_3(s)n^3+c_4(s)n^4.&
\end{alignat*}
We conclude that if $s\geq n^{2/3}$, then $h_n(s)>0$ and thus $M_{B_n(s-1)}>M_{B_n(s)}$.  Therefore,   $s_n<n^{2/3}$ for all $n$.
\end{proof}

For any $k>1$, we conclude that the limb weight of an optimal bridge of order $n$ is at least $\tfrac{4}{k}(n-2)^{2/3}$ for $n$ sufficiently large (depending on $k$).  This contrasts the situation for batons; the limb weight of an optimal baton of order $n$ is approximately $4\log_2(n).$  We can also use this fact to show that the mean subtree order of the optimal bridge of order $n$ must be significantly lower than the mean subtree order of the optimal baton of order $n$.  We use the following lemma which extends the idea of Result \ref{LeavesLemma}.

\begin{lemma}\label{TotalWeightLemma}
Let $T$ be a tree of order $n>3$ that is not a path and let the total limb weight of $T$ be $w.$  Then
\[
M_T \leq n-\tfrac{w}{2}.
\]
\end{lemma}

\begin{proof}
Recall that $M_{T,c(T)}$ denotes the average order of those subtrees of $T$ that contain the entire core $c(T)$.  By Results \ref{global_local} and \ref{nested_trees}(2), $M_{T,c(T)}\geq M_T.$  Note that when we contract $c(T)$ to a single new vertex $v$, the resulting tree is astral over $v$ and has order $w+1$.  Hence, by Result \ref{nested_trees}(1), we have
\[
M_{T,c(T)}=\tfrac{w+2}{2}+n-w-1=n-\tfrac{w}{2}. \qedhere
\]
\end{proof}

Thus, since the optimal bridge of order $n$ has total limb weight at least $\tfrac{4}{k}(n-2)^{2/3}$ for fixed $k>1$ and $n$ sufficiently large, it has mean subtree order at most $n-\tfrac{2}{k}(n-2)^{2/3}$ for $n$ sufficiently large, by Lemma \ref{TotalWeightLemma}.  This means that the optimal bridges have significantly lower mean subtree order than the corresponding optimal batons.

\section{A lower bound on the number of leaves in an optimal caterpillar}\label{CaterpillarLeaves}

In view of Jamison's Caterpillar Conjecture we consider here the structure of optimal caterpillars, i.e.\ trees that are optimal in $\C{n}$.  From Corollary \ref{LeavesCorollary}, we already know that if $T$ is optimal in $\C{n}$, then $T$ has at most $2\lceil\log_2 n\rceil+2$ leaves.  We show here that any tree optimal in $\C{n}$ must have at least roughly $\log_2(n)$ leaves.  We develop some general theory along the way which yields similar results for related families of trees.  Throughout, we assume that $n\geq 2$ so that every tree we consider has at least two leaves.  We begin with a simple definition.

\begin{definition}
For any tree $T$ of order $n\geq 2$, the tree obtained from $T$ by deleting all leaf vertices is called the \textit{stem} of $T.$
\end{definition}

Our first step is to bound the number of subtrees of a tree in terms of its number of leaves and the number of subtrees in its stem.

\begin{lemma}\label{StemBound}
Let $T$ be a tree with $\ell\leq n-2$ leaves and let $S$ be the stem of $T$.  Then
\[
N_T\leq N_S\cdot 2^\ell,
\]
where $N_T$ is the number of subtrees of $T$ and $N_S$ is the number of subtrees of $S.$
\end{lemma}

\begin{proof}
Let $\mathcal{C}(S)$ denote the collection of vertex sets of all subtrees of $S$, and likewise let $\mathcal{C}(T)$ denote the collection of vertex sets of all subtrees of $T.$  Let $L$ denote the set of leaves of $T.$  We show that there is an injection $\psi:\mathcal{C}(T)\rightarrow \mathcal{C}(S)\times \mathcal{P}(L)$, where $\mathcal{P}(L)$ denotes the power set of $L.$

Let $X\in \mathcal{C}(T).$  If $X\cap V(S)\neq \emptyset$, define
\[
\psi(X)=(X\cap V(S), X\cap L).
\]
Note that in this case the union of the components of $\psi(X)$ is $X$ since
\[
(X\cap V(S))\cup (X\cap L)=X\cap (V(S)\cup L)=X\cap V(T)=X.
\]
On the other hand, if $X\cap V(S)=\emptyset,$ then $X=\{v\}$ for some leaf $v\in L.$  Since $\ell\leq n-2$ there must be some vertex in $S$ that is not adjacent to $v.$  For each $v\in L$, fix a vertex $u_v$ in $S$ that is not adjacent to $v$ and define
\[
\psi(\{v\})=(\{u_v\},\{v\}).
\]
Note that in this case, the union of the components is not a member of $\mathcal{C}(T)$ as $v$ and $u_v$ are not adjacent in $T$.

Now let $X$ and $Y$ be distinct members of $\mathcal{C}(T).$  We show that $\psi(X)\neq \psi(Y).$  We have three cases:

\begin{enumerate}[label=\roman*)]
\item If $X\cap V(S)\neq \emptyset$ and $Y\cap V(S)\neq \emptyset$, then the union of the components of $\psi(X)$ is $X$ and the union of the components of $\psi(Y)$ is $Y.$  Since $X\neq Y$, it follows that $\psi(X)\neq \psi(Y).$

\item If $X\cap V(S)\neq \emptyset$ and $Y\cap V(S)=\emptyset,$ then the union of the components of $\psi(X)$ is $X\in \mathcal{C}(T),$ while the union of the components of $\psi(Y)$ is not in $\mathcal{C}(T)$.  It follows that $\psi(X)\neq \psi(Y).$

\item If $X\cap V(S)=Y\cap V(S)=\emptyset,$ then the second component of $\psi(X)$ is $X$ while the second component of $\psi(Y)$ is $Y.$  Since $X\neq Y$, it follows that $\psi(X)\neq \psi(Y).$
\end{enumerate}
We conclude that $\psi$ is injective, and hence
\[
N_T=|\mathcal{C}(T)|\leq |\mathcal{C}(S)|\cdot |\mathcal{P}(L)|=N_S\cdot 2^\ell. \qedhere
\]
\end{proof}

Written another way, the bound of Lemma \ref{StemBound} is $\tfrac{N_S}{N_T}\geq \tfrac{1}{2^\ell}$.  In words, we have a lower bound on the proportion of subtrees of $T$ that belong to the stem $S$ in terms of the number of leaves of $T$.  This leads to a bound on the mean subtree order of $T$ in terms of the number of leaves of $T$ and the mean subtree order of $S,$ obtained by considering the mean subtree order of $T$ as a weighted average.

\begin{theorem}\label{StemAndLeafBound}
Let $T$ be a tree of order $n$ with $\ell\leq n-2$ leaves and with stem $S$.  Then
\[
M_T\leq n-\tfrac{1}{2^\ell}(n-M_S).
\]
\end{theorem}

\begin{proof}
The subtrees of $T$ can be partitioned into two types: those that are contained entirely in $S$ and those that are not (i.e.\ those that contain a leaf of $T$).  Let $M_S$ denote the mean subtree order of $S$ and let $\overline{M_S}$ denote the mean order of those subtrees of $T$ that contain at least one leaf of $T$.  Expressing the mean subtree order of $T$ as a weighted average of $M_S$ and $\overline{M_S}$ gives us
\[
M_T=\frac{N_S}{N_T}\cdot M_S+\frac{N_T-N_S}{N_T}\cdot \overline{M_S}.
\]
We apply the trivial bound $\overline{M_S}\leq n$ and then the bound of Lemma \ref{StemBound} to obtain
\begin{align*}
M_T&\leq \frac{N_S}{N_T}\cdot M_S+\frac{N_T-N_S}{N_T}\cdot n\\
&=n-\frac{N_S}{N_T}\cdot (n-M_S)\\
&\le n-\tfrac{1}{2^\ell}\cdot (n-M_S). \qedhere
\end{align*}
\end{proof}

We see that the bound of Theorem \ref{StemAndLeafBound} gives us more information when $\ell$ and $M_S$ are small relative to the order $n$ of $T$.  Note that the stem of every caterpillar is a path, and the path $P_n$ is known to have minimum mean subtree order in $\T{n}$.  Thus, we expect Theorem \ref{StemAndLeafBound} to give a fairly effective bound on the mean subtree order of any caterpillar, at least in the case that the caterpillar has very few leaves.  We use this idea along with Corollary \ref{MaxMeanCorollary} to prove that if a caterpillar is optimal in $\T{n}$ (or $\C{n}$) then it must have at least $\log_2(n)-\log_2(\log_2(n)+1)-\log_2(3)$ leaves.

\begin{corollary}\label{CaterpillarLeaf}
If $T$ is a caterpillar of order $n$ with $\ell$ leaves and
\[
\ell\leq \log_2\left(\frac{n}{3\log_2(n)+3}\right)=\log_2(n)-\log_2(\log_2(n)+1)-\log_2(3),
\]
then $T$ is not optimal in $\T{n}$ or $\C{n}$.
\end{corollary}

\begin{proof}
Let $T$ be a caterpillar of order $n$ with $\ell\leq \log_2\left(\tfrac{n}{3\log_2(n)+3}\right)$ leaves.  Since $\ell<\log_2\left(\tfrac{n}{3}\right)\leq n-2$, we may apply Theorem \ref{StemAndLeafBound}.  Moreover, since $T$ is a caterpillar, the stem $S$ of $T$ is a path of order $n-\ell.$  Thus $M_S=\tfrac{n-\ell+2}{3}\leq \tfrac{n}{3}$, by Result \ref{paths}.  By Theorem \ref{StemAndLeafBound},
\begin{align*}
M_T\leq n-\tfrac{1}{2^\ell}\cdot (n-M_S).
\end{align*}
The fact that $M_S\leq \tfrac{n}{3}$ and the assumption that $\ell\leq \log_2\left(\tfrac{n}{3\log_2(n)+3}\right)$ yields
\[
M_T\leq n-\tfrac{3\log_2(n)+3}{n}\cdot \tfrac{2}{3}n=n-2\log_2(n)-2.
\]
By Corollary \ref{MaxMeanCorollary}, $T$ cannot be optimal in $\T{n}$ or $\C{n}.$
\end{proof}

Corollary \ref{CaterpillarLeaf} is a particular case of a more general result, stated below.  As long as the stem of a tree $T$ belongs to a family of trees of density at most $k$ for some $k<1$, we can show that if $T$ has too few leaves, then it is not an optimal tree among all trees of order $n.$

\begin{corollary}\label{GeneralLeaf}
Fix any real number $k\in\left[\tfrac{1}{3},1\right)$, and let $\mathcal{S}_k$ be a family of trees such that if $S\in\mathcal{S}_k$ then $M_S\leq k(|V(S)|+2)$ (in particular, the family of trees of density at most $k$ is one such family).  If $T$ is a tree of order $n$ with $\ell$ leaves whose stem $S$ belongs to $\mathcal{S}_k$ and
\[
\ell\leq \log_2\left(\frac{(1-k)n}{2\log_2(n)+2}\right)=\log_2(n)-\log_2(\log_2(n)+1)-\log_2\left(\tfrac{2}{1-k}\right),
\]
then $T$ is not optimal in $\T{n}.$
\end{corollary}

\begin{proof}
Let $T$ be a tree of order $n$ with $\ell \leq \log_2\left(\frac{(1-k)n}{2\log_2(n)+2}\right)$
leaves and with stem $S\in\mathcal{S}_k$, where $\mathcal{S}_k$ is as above.  Note that $\ell<\log_2\left(\tfrac{n}{3}\right)<n-2,$ so that we may apply Theorem \ref{StemAndLeafBound}.  From the definition of $\mathcal{S}_k$, $M_S\leq k(|V(S)|+2)\leq kn$.  Thus, by Theorem \ref{StemAndLeafBound},
\[
M_T\leq n-\tfrac{1}{2^\ell}\cdot (n-M_S)\leq n-\frac{2\log_2(n)+2}{(1-k)n}\cdot (1-k)n=n-2\log_2(n)-2.
\]
Hence, by Corollary \ref{MaxMeanCorollary}, $T$ cannot be optimal in $\T{n}.$
\end{proof}

Note that there are several obvious families (in addition to the caterpillars) to which Corollary \ref{GeneralLeaf} can be applied.  A family that comes to mind immediately is the collection of trees whose stems are asters; if $S$ is astral over $v$, then $M_S\leq M_{S,v}=\tfrac{|V(S)|+1}{2}< \tfrac{|V(S)|+2}{2}$.  From Corollary \ref{GeneralLeaf}, we conclude that if $T$ is a tree of order $n$ with $\ell$ leaves whose stem is an aster and $\ell\leq\log_2\left(\tfrac{n}{4\log_2(n)+4}\right)$, then $T$ is not optimal in $\T{n}$.  Another example is the collection of trees whose stems are series-reduced trees.  Series-reduced trees were shown to have density at most $\tfrac{3}{4}$ in \cite{vw}.  Therefore, if $T$ is a tree of order $n$ whose stem is a series-reduced tree and $T$ has at most $\log_2\left(\tfrac{n}{8\log_2(n)+8}\right)$ leaves, then $T$ is not optimal in $\T{n}.$

\section{Concluding remarks}

In this article we established the Gluing Lemma which allowed us to determine optimal trees in several families.  Our work on the Gluing Lemma led to a proof that the limbs of any tree optimal in $\T{n}$ all have order $1$, and to an answer to an open problem of Jamsion \cite{j1}.  We showed that among all subdivided double stars of order n with an even (and sufficiently large) number of leaves, the batons are optimal.  We described the asymptotic structure of any optimal tree in the family of all batons of a fixed order and any optimal tree in the family of all bridges of a fixed order.  While Jamison's Caterpillar Conjecture remains open, we demonstrated that the number of leaves in an optimal tree in $\C{n}$ is $\Theta(\log_2 n)$.  It remains an open problem to determine whether the number of leaves in an optimal tree in $\T{n}$ is $\Theta(\log_2 n)$, but we have shown that it is $\mathrm{O}(\log_2 n).$

\section*{Acknowledgements}

We would like to thank the anonymous referees for their excellent comments, which helped to improve the article.  We would also like to thank Wayne Goddard for computing the optimal tree in $\T{19}$ and $\T{20}$, which encouraged us to determine the optimal trees for the next four orders.  Our computation of the optimal tree in $\T{n}$ for $21\leq n\leq 24$ was enabled in part by support provided by WestGrid (\url{www.westgrid.ca}) and Compute Canada (\url{www.computecanada.ca}).

%\bibliographystyle{amsplain}
%\bibliography{MeansBib}

\begin{thebibliography}{1}

\bibitem{has}
J.~Haslegrave, \emph{Extremal results on average subtree density of
  series-reduced trees}, J. Combin. Theory Ser. B \textbf{107} (2014), 26--41.

\bibitem{j1}
R.~E. Jamison, \emph{On the average number of nodes in a subtree of a tree}, J.
  Combin. Theory Ser. B \textbf{35} (1983), 207--223.

\bibitem{j2}
R.~E. Jamison, \emph{Monotonicity of the mean order of subtrees}, J. Combin. Theory
  Ser. B \textbf{37} (1984), 70--78.

\bibitem{mp}
B.~D. McKay and A. Piperno, \emph{Practical Graph Isomorphism, II},
J. Symbolic Comput. \textbf{60} (2014), 94--112.

\bibitem{vw}
A.~Vince and H.~Wang, \emph{The average order of a subtree of a tree}, J.
  Combin. Theory Ser. B \textbf{100} (2010), 161--170.

\bibitem{Wagner2007}
S.~Wagner, \emph{Correlation of graph-theoretical indices}, SIAM J. Discr. Math. \textbf{21} (2007), 33--46.

\bibitem{ww1}
S.~Wagner and H.~Wang, \emph{Indistinguishable trees and graphs}, Graphs
  Combin. \textbf{30} (2014), 1593--1605.

\bibitem{ww}
S.~Wagner and H.~Wang, \emph{On the local and global means of subtree orders}, J. Graph
  Theory \textbf{81}(2) (2016), 154--166.

\bibitem{yy}
W. Yan and Y.-N. Yeh, \emph{Enumeration of subtrees of trees}, Theoret. Comput. Sci. \textbf{369}(1-3) (2006), 256--268.

\end{thebibliography}

\providecommand{\bysame}{\leavevmode\hbox to3em{\hrulefill}\thinspace}
\providecommand{\MR}{\relax\ifhmode\unskip\space\fi MR }
\providecommand{\MRhref}[2]{%
  \href{http://www.ams.org/mathscinet-getitem?mr=#1}{#2}
}
\providecommand{\href}[2]{#2}

\appendix

\section{Demonstrating that $h_{n,m}(s)>0$ for $1\leq s\leq m-1$}

Recall that
\begin{equation}
\begin{split}\label{ha}
h_{n,m}(s)&=\overline{A}_{n,m}\left[B_{n,m}(s)-B_{n,m}(m)\right]-A_{n,m}\left[\overline{B}_{n,m}(s)-\overline{B}_{n,m}(m)\right]\\
& \ \ \ \ -\left[\overline{B}_{n,m}(s)\cdot B_{n,m}(m)-\overline{B}_{n,m}(m)\cdot B_{n,m}(s)\right],
\end{split}
\end{equation}
where
\begin{align}
A_{n,m}&=2m+\tbinom{n-2m-1}{2}+2^{2m},\label{A}\\
\overline{A}_{n,m}&=2m+\tbinom{n-2m}{3}+(n-m)\cdot 2^{2m},\label{Abar}\\
B_{n,m}(s)&=(n-2m-1)\cdot (2^s+2^{2m-s}), \mbox{ and}\label{B}\\
\overline{B}_{n,m}(s)&=\tfrac{1}{2}(n-2m-1)\left[(n-2m+s)\cdot 2^s+(n-s)\cdot 2^{2m-s}\right].\label{BBar}
\end{align}

We first consider each of the three bracketed expressions in (\ref{ha}) separately.  Substituting the expressions given by (\ref{B}) and (\ref{BBar}), and then factoring, we obtain
\begin{align*}
B_{n,m}(s)-B_{n,m}(m)&=(n-2m-1)\cdot (2^s+2^{2m-s})-2(n-2m-1)\cdot 2^m\\
&=\tfrac{1}{2^s}(n-2m-1)(2^{2s}-2\cdot 2^{m+s}+2^{2m})\\
&=\tfrac{1}{2^s}(n-2m-1)(2^m-2^s)^2,
\end{align*}
\begin{align*}
\overline{B}&_{n,m}(s)-\overline{B}_{n,m}(m)\\
&=\tfrac{1}{2}(n-2m-1)\left[(n-2m+s)\cdot 2^s+(n-s)\cdot 2^{2m-s}\right]\\
&\ \ \ \ -\tfrac{1}{2}(n-2m-1)\left[2(n-m)\cdot 2^m\right]\\
&=\tfrac{1}{2} (n-2m-1)\left[n(2^s+2^{2m-s}-2\cdot 2^{m})+2m(2^m-2^s)-s(2^{2m-s}-2^s) \right]\\
&= \tfrac{1}{2} (n-2m-1)\left[\tfrac{n}{2^s}(2^m-2^s)^2+2m(2^m-2^s)-\tfrac{s}{2^s}(2^m-2^s)(2^m+2^s)\right]\\
&=\tfrac{1}{2}\cdot \tfrac{1}{2^s}(n-2m-1)(2^m-2^s)\left[n(2^m-2^s)+2m\cdot 2^s-s(2^m+2^s)\right]\\
&=\tfrac{1}{2}\cdot \tfrac{1}{2^s}(n-2m-1)(2^m-2^s)\left[n(2^m-2^s)+2(m-s)\cdot 2^s-s(2^m-2^s)\right]\\
&=\tfrac{1}{2}\cdot \tfrac{1}{2^s}(n-2m-1)(2^m-2^s)^2\left[n+2(m-s)\cdot \tfrac{2^s}{2^m-2^s}-s\right],
\end{align*}
and
\begin{align*}
&\overline{B}_{n,m}(s)\cdot B_{n,m}(m)-\overline{B}_{n,m}(m)\cdot B_{n,m}(s)\\
& \ =\tfrac{1}{2}(n-2m-1)\left[(n-2m+s)\cdot 2^s+(n-s)\cdot 2^{2m-s}\right]\cdot 2(n-2m-1)\cdot 2^m\\
& \  \ \ \ \ -\tfrac{1}{2}(n-2m-1)\left[2(n-m)\cdot 2^m\right]\cdot \left[(n-2m-1)\cdot (2^s+2^{2m-s})\right]\\
& \ =(n-2m-1)^2\cdot 2^m\cdot \bigl[(n-2m+s)\cdot 2^s+(n-s)\cdot 2^{2m-s}\\
& \  \hspace{6.5cm}-(n-m)\cdot (2^s+2^{2m-s})\bigr]\\
& \ =(n-2m-1)^2\cdot 2^m\cdot \left[-2m\cdot 2^s+s\cdot 2^s-s\cdot 2^{2m-s}+m\cdot (2^s+2^{2m-s})\right]\\
& \ =(n-2m-1)^2\cdot 2^m\cdot \left[(m-s)\cdot 2^{2m-s}+(s-m)\cdot 2^s\right]\\
& \ =(n-2m-1)^2(m-s)(2^{2m-s}-2^s)\cdot 2^m\\
& \ =\tfrac{1}{2^s}(n-2m-1)^2(m-s)(2^m-2^s)(2^m+2^s)\cdot 2^m\\
& \ =\tfrac{1}{2^s}(n-2m-1)^2(2^m-2^s)^2(m-s)\cdot \tfrac{2^m+2^s}{2^m-2^s}\cdot 2^m.
\end{align*}
Note that we have written these three expressions so that they each have a factor of $\tfrac{1}{2^s}(n-2m-1)(2^m-2^s)^2$, which is clearly positive.  Substituting back into (\ref{ha}), we obtain
\begin{align*}
h_{n,m}(s)&=\tfrac{1}{2^s}(n-2m-1)(2^m-2^s)^2 g_{n,m}(s),
\end{align*}
where
\begin{align*}
g_{n,m}(s)&=\overline{A}_{n,m}-\tfrac{1}{2}\left[n+2(m-s)\cdot\tfrac{2^s}{2^m-2^s}-s\right]A_{n,m}\\
& \ \ \ \ -(n-2m-1)(m-s)\cdot \tfrac{2^m+2^s}{2^m-2^s}\cdot 2^m.
\end{align*}
Thus, it suffices to show that $g_{n,m}(s)>0$ for $1\leq s\leq m-1.$

First we claim that
\begin{align}\label{ineq}
(m-s)\cdot \tfrac{2^s}{2^m-2^s}\leq 1
\end{align}
for all $s\leq m-1$.  Setting $k=m-s$ (note that $k\geq 1$ since $s\leq m-1$), we have
\[
(m-s)\cdot \tfrac{2^s}{2^m-2^s}=(m-s)\cdot\tfrac{1}{2^{m-s}-1}=\tfrac{k}{2^k-1},
\]
and $k\leq 2^k-1$ is easily verified for all $k\geq 1$ by induction, which completes the proof of the claim.  We apply (\ref{ineq}) along with $s\geq 1$ to obtain:
\begin{align*}
n+2(m-s)\cdot \tfrac{2^s}{2^m-2^s}-s\leq n+2-s\leq n+1
\end{align*}
and
\begin{align*}
(m-s)\cdot \frac{2^m+2^s}{2^m-2^s}&= (m-s)\cdot \frac{2^m-2^s+2\cdot 2^s}{2^m-2^s}\\
&= (m-s)+2(m-s)\cdot\frac{2^s}{2^m-2^{s}}\\
&\leq (m-s)+2\\
&\leq m+1.
\end{align*}

\medskip

\noindent For the remainder of the proof we consider two cases.

\medskip

\noindent{\bf Case 1:} $n > 2m+2$.

In this case,
\[
\overline{A}_{n,m}=2m+\tbinom{n-2m}{3}+(n-m)\cdot 2^{2m}\geq (n-m)\cdot 2^{2m}
\]
and
\[
A_{n,m}=2m+\tbinom{n-2m-1}{2}+2^{2m}\leq \tfrac{(n-1)(n-2m-1)}{2}+2^{2m},
\]
the former being obvious and the latter easily verified by expanding the binomial coefficient.  Applying all of the above inequalities, we find
\begin{align*}
g_{n,m}(s)&=\overline{A}_{n,m}-\tfrac{1}{2}\left[n+2(m-s)\cdot\tfrac{2^s}{2^m-2^s}-s\right]A_{n,m}\\
& \ \ \ \ -(n-2m-1)(m-s)\cdot \tfrac{2^m+2^s}{2^m-2^s}\cdot 2^m\\
&\geq (n-m)\cdot 2^{2m}-\tfrac{1}{2}(n+1)\left[\tfrac{(n-1)(n-2m-1)}{2}+2^{2m}\right]\\
& \ \ \ \ -(m+1)(n-2m-1)\cdot 2^m\\
&=\tfrac{1}{2}(n-2m-1)\cdot 2^{2m}-\tfrac{1}{2}(n+1)\tfrac{(n-1)(n-2m-1)}{2}\\
& \ \ \ \ -(m+1)(n-2m-1)\cdot 2^m\\
&=\tfrac{1}{2}(n-2m-1)\left[2^{2m}-\tfrac{1}{2}(n+1)(n-1)-2(m+1)\cdot 2^m\right]\\
&> \tfrac{1}{2}(n-2m-1)\left[2^{2m}-\tfrac{1}{2}n^2-2(m+1)\cdot 2^m\right],
\end{align*}
where the last inequality follows from the fact that
\[
(n+1)(n-1)=n^2-1< n^2.
\]
Finally, using the assumptions $m\geq \log_2(n)$ and $m\geq 5$, the latter of which implies $2^m> 4(m+1)$ (verified by induction), we have
\begin{align*}
2^{2m}-\tfrac{1}{2}n^2-2(m+1)\cdot 2^m&=\tfrac{1}{2}\cdot 2^{2m}-\tfrac{1}{2}n^2+\tfrac{1}{2}\cdot 2^{2m}-2(m+1)\cdot 2^m\\
&> \tfrac{1}{2}n^2-\tfrac{1}{2}n^2+\tfrac{1}{2}\cdot 4(m+1)\cdot 2^m-2(m+1)\cdot 2^m\\
&=0.
\end{align*}
Therefore, $g_{n,m}(s)>0$ and hence $h_{n,m}(s)>0$ for $1\leq s\leq m-1$ when $2m+2<n$.\\

\noindent{\bf Case 2:} $n=2m+2$.

In this case,
\[
\overline{A}_{n,m}=2m+(m+2)\cdot 2^{2m} \mbox{ and } A_{n,m}=2m+2^{2m}.
\]
Thus,
\begin{align*}
g_{n,m}(s)&=2m+(m+2)2^{2m} -\tfrac{1}{2}\left[2m+2+2(m-s)\cdot\tfrac{2^s}{2^m-2^s}-s\right]\left[2m+2^{2m} \right]\\
& \ \ \ \ -(m-s)\cdot \tfrac{2^m+2^s}{2^m-2^s}\cdot 2^m\\
&\ge 2m+(m+2)2^{2m} -\tfrac{1}{2}\left[2m+2+2-s\right]\left[2m+2^{2m} \right]-(m+1)\cdot 2^m\\
&\ge 2m+(m+2)2^{2m} -\tfrac{1}{2}\left[2m+3\right]\left[2m+2^{2m} \right]-(m+1)\cdot 2^m\\
&=2^{2m-1}-2m^2-m-(m+1)2^m\\
&=\left[2^{2m-2}-(2m^2+m)\right]+\left[2^{2m-2}-(m+1)2^m\right]\\
&=\left[2^{2m-2}-m(2m+1)\right]+2^m\cdot \left[2^{m-2}-(m+1)\right]\\
&>0
\end{align*}

\noindent where in the last inequality we use the facts that $2^{2m-2}-m(2m+1)>0$ and $2^{m-2}-(m+1) >0$ for $m\geq 5.$  Both of these inequalities can be verified by induction.  Thus, $g_{m,n}(s)>0$ in the case that $n=2m+2$ as well.  This completes the proof.

\section{Bounding the coefficient functions of $f_n(s)$} \label{IneqApp}

Recall that
\begin{align*}
c_4(s)&=\tfrac{1}{6}\cdot 2^s-\tfrac{1}{6},\\
c_3(s)&=2^{2s}-\tfrac{1}{6}(5s+14)\cdot 2^s+\tfrac{4}{3}(s+1),\\
c_2(s)&=2\cdot 2^{3s}-\tfrac{1}{2}(9s+16)\cdot 2^{2s}\\
& \ \ \ \ +\tfrac{1}{6}(6s^2+84s+59)\cdot 2^s-\tfrac{1}{6}(24s^2+60s+23,)\\
c_1(s)&=(6s+8)\cdot 2^{3s}-\tfrac{1}{2}(12s^2+71s+42)\cdot 2^{2s}\\
& \ \ \ \ -\tfrac{1}{6}(4s^3-156s^2-347s-106)\cdot 2^s\\
& \ \ \ \ -\tfrac{1}{3}(16s^3+72s^2+82s+14), \mbox{ and}\\
c_0(s)&=4\cdot 2^{4s}-(4s^2+18s+14)\cdot 2^{3s}+(2s^3+33s^2+60s+18)\cdot 2^{2s}\\
& \ \ \ \ +\tfrac{1}{3}(4s^4-44s^3-199s^2-208s-30)\cdot 2^s\\
& \ \ \ \ +\tfrac{1}{3}(8s^4+56s^3+118s^2+82s+6).
\end{align*}

For all $s\geq 1$ we have $2^s\geq 2s$, which is easily verified by induction.  We use this inequality frequently in the following arguments.

\begin{enumerate}[label=\roman*)]
\item $c_4(s)<\tfrac{1}{6}\cdot 2^s$ for $s\geq 1$.

This inequality is obvious.

\item $c_3(s)<2^{2s}$ for $s\geq 1$

Let $s\geq 1.$  Expanding and then applying the inequality $2^s\geq 2s$, and finally applying the inequality $s\geq 1$ itself, we have
\begin{align*}
c_3(s)&=2^{2s}-\tfrac{5}{6}s\cdot 2^s-\tfrac{7}{3}\cdot 2^s+\tfrac{4}{3}s+1\\
&<2^{2s}-\tfrac{5}{6}\cdot 2^s-\tfrac{14}{3}s+\tfrac{4}{3}s+1\\
&<2^{2s}-\tfrac{5}{6}s\cdot 2^s-\tfrac{7}{3}s\\
&<2^{2s}
\end{align*}
as desired.

\item $c_2(s)<2\cdot 2^{3s}$ for $s\geq 1.$

Let $s\geq 1.$  We immediately drop the last term of $c_2(s)$ (since it is negative) and then once again apply the inequality $2^{s}\geq 2s$.
\begin{align*}
c_2(s)&<2\cdot 2^{3s}-\tfrac{1}{2}(9s+16)\cdot 2^{2s}+\tfrac{1}{6}(6s^2+84s+59)\cdot 2^s\\
&\leq 2\cdot 2^{3s}-(9s^2+16s-s^2-14s-\tfrac{59}{6})\cdot 2^s\\
&=2\cdot 2^{3s}-(8s^2+2s-\tfrac{59}{6})\cdot 2^s\\
&<2\cdot 2^{3s},
\end{align*}
where the last inequality follows since $8s^2+2s-\tfrac{59}{6}>0$ for $s\geq 1.$

\item $c_1(s)<(6s+8)\cdot 2^{3s}$ for $s\geq 1.$

Let $s\geq 1.$  We immediately drop the final term of $c_1(s)$ (which is negative), and again use the inequality $2^s\geq 2s.$
\begin{align*}
c_1(s)&< (6s+8)\cdot 2^{3s}-\tfrac{1}{2}(12s^2+71s+42)\cdot 2^{2s}\\
& \ \ \ \ -\tfrac{1}{6}(4s^3-156s^2-347s-106)\cdot 2^s\\
&\leq (6s+8)\cdot 2^{3s}-(12s^3+71s^2+42s+\tfrac{2}{3}s^3-26s^2-\tfrac{347}{6}s-\tfrac{53}{3})\cdot 2^s\\
&=(6s+8)\cdot 2^{3s}-(\tfrac{38}{3}s^3+50s^2-\tfrac{95}{6}s-\tfrac{53}{3})\cdot 2^s\\
&<(6s+8)\cdot 2^{3s},
\end{align*}
where the last inequality follows from the fact that $\tfrac{38}{3}s^3+50s^2-\tfrac{95}{6}s-\tfrac{53}{3}>0$ for $s\geq 1.$

\item $c_0(s)<4\cdot 2^{4s}$ for $s\geq 1.$

For $s=1$ and $s=2,$ one can verify this bound directly.  Now let $s\geq 3$.  Regrouping and then applying the inequality $2^s\geq 2s$, we have
\begin{align*}
c_0(s)&=4\cdot 2^{4s}-(4s^2+18s+14)\cdot 2^{3s}+(2s^3+33s^2+60s+18)\cdot 2^{2s}\\
& \ \ \ \ +\tfrac{4}{3}s^4\cdot 2^s-\tfrac{1}{3}(44s^3+199s^2+208s+30)\cdot 2^s\\
& \ \ \ \ +\tfrac{1}{3}(8s^4+56s^3+118s^2+82s+6)\\
&\leq 4\cdot 2^{4s}-(8s^3+36s^2+28s-2s^3-33s^2-60s-18)\cdot 2^{2s}+\tfrac{4}{3}s^4\cdot 2^s\\
& \ \ \ \ -\tfrac{1}{3}(44s^4+199s^3+208s^2+30s-8s^4-56s^3-118s^2-82s-6)\\
&=4\cdot 2^{4s}-(6s^3+3s^2-32s-18)\cdot 2^{2s}+\tfrac{4}{3}s^4\cdot 2^s \\
& \ \ \ \ -\tfrac{1}{3}(36s^4+143s^3+90s^2-52s-6).
\end{align*}
At this point, we note that $36s^4+143s^3+90s^2-52s-6>0$.  Thus, we can drop the final term in the expression above.  We also break off part of the second term to take care of the third (using the inequality $2^s\geq 2s$ once again).  This gives
\begin{align*}
c_0(s)&<4\cdot 2^{4s}-(\tfrac{16}{3}s^3+3s^2-32s-18)\cdot 2^{2s}-\tfrac{2}{3}s^3\cdot 2^{2s}+\tfrac{4}{3}s^4\cdot 2^s\\
&\leq 4\cdot 2^{4s}-(5s^3+3s^2-32s-18)\cdot 2^{2s}-\tfrac{4}{3}s^4\cdot 2^s+\tfrac{4}{3}s^4\cdot 2^s\\
&=4\cdot 2^{4s}-(5s^3+3s^2-32s-18)\cdot 2^{2s}\\
&<4\cdot 2^{4s},
\end{align*}
where the last inequality follows from the fact that $5s^3+3s^2-32s-18>0$ for $s\geq 3$ (this is easily verified).

\item $c_4(s)>0$ for $s\geq 1$.

This inequality is obvious.

\item $c_3(s)>0$ for $s\geq 2$.

Let $s\geq 2$.  We drop the third term of $c_3(s)$ immediately, and then apply the inequality $2^s\geq 2s$ in the first term:
\begin{align*}
c_3(s)&>2^{2s}-\tfrac{1}{6}(5s+14)\cdot 2^s\\
&\geq \left(2s-\tfrac{5}{6}s-\tfrac{14}{6}\right)\cdot 2^s\\
&=\tfrac{7}{6}(s-2)\cdot 2^s\\
&\geq 0.
\end{align*}

\item $c_2(s)>\tfrac{5}{3}\cdot 2^{3s}$ for $s\geq 7.$

Let $s\geq 7$.  We show that $c_2(s)-\tfrac{5}{3}\cdot 2^{3s}> 0$.  Note that for $s\geq 7,$ we have $2^s> \tfrac{27}{2}s+24$ (verified by induction), and clearly we have $2^s\geq 4$.  Applying these inequalities, we have
\begin{align*}
c_2(s)-\tfrac{5}{3}\cdot 2^{3s}&=\tfrac{1}{3}\cdot 2^{3s}-\tfrac{1}{2}(9s+16)\cdot 2^{2s}\\
& \ \ \ \ +\tfrac{1}{6}(6s^2+84s+59)\cdot 2^s-\tfrac{1}{6}(24s^2+60s+23)\\
&>\tfrac{1}{3}(\tfrac{27}{2}s+24)\cdot 2^{2s} -\tfrac{1}{2}(9s+16)\cdot 2^{2s}\\
& \ \ \ \ +\tfrac{1}{6}(24s^2+336s+236)-\tfrac{1}{6}(24s^2+60s+23)\\
&=\tfrac{1}{6}(276s+213)\\
&>0.
\end{align*}

\item $c_1(s)> 0$ for $s\geq 5$.

Let $s\geq 5$.  Using the inequality $2^s\geq 2s$ in the first term, we obtain:
\begin{align*}
c_1(s)&=(6s+8)\cdot 2^{3s}-\tfrac{1}{2}(12s^2+71s+42)\cdot 2^{2s}\\
& \ \ \ \ -\tfrac{1}{6}(4s^3-156s^2-347s-106)\cdot 2^s -\tfrac{1}{3}(16s^3+72s^2+82s+14)\\
&\geq \left(12s^2+16s-6s^2-\tfrac{71}{2}s-21\right)\cdot 2^{2s}\\
& \ \ \ \ -\tfrac{1}{6}(4s^3-156s^2-347s-106)\cdot 2^s -\tfrac{1}{3}(16s^3+72s^2+82s+14)\\
&= \left(6s^2-\tfrac{39}{2}s-21\right)\cdot 2^{2s}\\
& \ \ \ \ -\tfrac{1}{6}(4s^3-156s^2-347s-106)\cdot 2^s -\tfrac{1}{3}(16s^3+72s^2+82s+14)
\end{align*}
We verify that $6s^2-\tfrac{39}{2}s-21>0$ for $s\geq 5$, so that we may apply the inequality of $2^s\geq 2s$ to the first term once again:
\begin{align*}
c_1(s)&\geq \left(6s^2-\tfrac{39}{2}s-21\right)\cdot 2^{2s}\\
& \ \ \ \ -\tfrac{1}{6}(4s^3-156s^2-347s-106)\cdot 2^s -\tfrac{1}{3}(16s^3+72s^2+82s+14)\\
&\geq \left(12s^2-39s^2-42s-\tfrac{4}{6}s^3+26s^2+\tfrac{347}{6}s+\tfrac{53}{3}\right)\cdot 2^{s}\\
& \ \ \ \  -\tfrac{1}{3}(16s^3+72s^2+82s+14)\\
&=\left(\tfrac{34}{3}s^3-13s^2+\tfrac{95}{6}s+\tfrac{53}{3}\right)\cdot 2^{s}-\tfrac{1}{3}(16s^3+72s^2+82s+14)
\end{align*}
We verify that $\tfrac{34}{3}s^3-13s^2+\tfrac{95}{6}s+\tfrac{53}{3}>0$ for $s\geq 5$, so that we may apply $2^s\geq 2s$ once more:
\begin{align*}
c_1(s)&\geq \left(\tfrac{34}{3}s^3-13s^2+\tfrac{95}{6}s+\tfrac{53}{3}\right)\cdot 2^{s}-\tfrac{1}{3}(16s^3+72s^2+82s+14)\\
&\geq \tfrac{68}{3}s^4-26s^3+\tfrac{95}{3}s^2+\tfrac{106}{3}s-\tfrac{16}{3}s^3-24s^2-\tfrac{82}{3}s-\tfrac{14}{3}\\
&\geq \tfrac{68}{3}s^4-\tfrac{94}{3}s^3+\tfrac{23}{3}s^2+8s-\tfrac{14}{3},
\end{align*}
which is easily verified to be positive for $s\geq 5$.

\item $c_0(s)> \tfrac{19}{6}\cdot 2^{4s}$ for $s\geq 10.$

Let $s\geq 10$.  We show equivalently that $c_0(s)-\tfrac{19}{6}\cdot 2^{4s}>0$.  We may immediately drop the final term of $c_0(s)$ since it is positive, so we have
\begin{align*}
c_0(s)-\tfrac{19}{6}\cdot 2^{4s}&>\tfrac{5}{6}\cdot 2^{4s}-(4s^2+18s+14)\cdot 2^{3s}\\
& \ \ \ \ +(2s^3+33s^2+60s+18)\cdot 2^{2s}\\
& \ \ \ \ +\tfrac{1}{3}(4s^4-44s^3-199s^2-208s-30)\cdot 2^s.
\end{align*}
Now we use the inequality
\[
\tfrac{5}{6}\cdot 2^s>4s^2+18s+14,
\]
which is verified for $s\geq 10$ by induction.  This gives
\begin{align*}
c_0(s)-\tfrac{19}{6}\cdot 2^{4s}&>(2s^3+33s^2+60s+18)\cdot 2^{2s}\\
& \ \ \ \ +\tfrac{1}{3}(4s^4-44s^3-199s^2-208s-30)\cdot 2^s.
\end{align*}
Finally, we apply the inequality $2^s\geq 2s$ to obtain
\begin{align*}
c_0(s)-\tfrac{19}{6}\cdot 2^{4s}&> (4s^4+66s^3+120s^2+36s)\cdot 2^s\\
& \ \ \ \ +\tfrac{1}{3}(4s^4-44s^3-199s^2-208s-30)\cdot 2^s\\
&=\tfrac{1}{3}(16s^4+154s^3+161s^2-100s-30)\\
&>0.
\end{align*}

\end{enumerate}
\end{document}